\documentclass[10pt]{amsart}
\usepackage[utf8]{inputenc}
\usepackage[centertags]{amsmath}
\usepackage{amsfonts}
\usepackage{amscd}
\usepackage{amsthm}
\usepackage{newlfont}
\usepackage{amscd}
\usepackage{amsgen}
\usepackage{amssymb}
\usepackage{fullpage}
\usepackage{longtable}
\usepackage{listings}
\usepackage{easybmat}
\usepackage{hyperref}

\theoremstyle{plain} \newtheorem{thm}{Theorem}[section]
\newtheorem{prob}[thm]{Problem}
\newtheorem{lemm}[thm]{Lemma}
\newtheorem{prop}[thm]{Proposition}
\newtheorem{cor}[thm]{Corollary}
\theoremstyle{definition}
\newtheorem{defi}[thm]{Definition}
\newtheorem{rem}[thm]{Remark}

\newtheorem{example}[thm]{Example}

 \numberwithin{equation}{section}
\numberwithin{equation}{section}

 \DeclareMathOperator{\coker}{coker}

\newcommand\PP{{\mathbb{P}}}

            \newcommand\oo{\mathcal O}
\theoremstyle{definition}

\newtheorem{conj}[thm]{Conjecture}
\title{Tonoli's Calabi--Yau threefolds revisited}
\author{Grzegorz Kapustka and Michal Kapustka}

\begin{document}
\setcounter{tocdepth}{1}
\begin{abstract} We find a simple construction of Tonoli's examples of
Calabi--Yau threefolds in complex
$\mathbb{P}^6$. We prove that the rank of the Picard group of elements of one of
these families is at least $2$.
\end{abstract}
\maketitle

\section{Introduction}
From the famous Serre construction, we know that a codimension $2$ submanifold
$X\subset \mathbb{P}^{n}$ that is subcanonical
(i.e.~$\omega_X\simeq\mathcal{O}_X(l)$ for some $l\in\mathbb{Z}$) can be seen as
the zero locus of a section of a rank two vector bundle $E$. In particular, in
that case, we have
an exact sequence $$0\to\mathcal{O}_{\mathbb{P}^n}\to E\to
\mathcal{I}_X(c_1(E))\to 0.$$

There is a similar construction in codimension $3$. By answering
Okonek's question (see \cite{Okonek}), Walter showed in \cite{Walter} (see also
\cite{EPW}) that if
$n-3$ is not divisible
by $4$ then each locally Gorenstein subcanonical subscheme of codimension $3$
in $\mathbb{P}^n$
admits a Pfaffian resolution
\[ 0\to\mathcal{O}_{\mathbb{P}^n}(-2s-t)\to E^{\ast}(-s-t)\xrightarrow{\varphi}
E(-s)\to\mathcal{I}_X\to 0, \]
 where the vector bundle $E$ is of rank $2r+1$ and
$s=c_1(E)+rt$. Moreover, in that case
\begin{equation}\label{formula for canonical}
\omega_X=\mathcal{O}_X(t+2s-n-1). 
\end{equation}
 In such a situation, we shall say that $E$
defines
$X$ through the Pfaffian construction, although $X$ is in fact defined by
an additional choice of a map $\varphi:E^{\ast}(-s-t)\rightarrow
E(-s)$ or equivalently by a choice of a section $\sigma \in (\bigwedge^2 E)(t)$. When 
such a $\sigma$ is specified we shall write $\operatorname{Pf}(\sigma)$ for the Pfaffian variety defined 
by $\sigma$.

The Pfaffian construction was applied in \cite{Ca} to construct canonically
embedded surfaces of general type in $\mathbb{P}^5$, and in \cite{Tonoli} and
\cite{SchreyerTonoli} to construct
Calabi--Yau threefolds in $\mathbb{P}^6$. The latter examples will be referred to
as Tonoli Calabi--Yau threefolds. The resulting examples of both papers have
degrees $12\leq d\leq 17$. In
particular, in the case of degree $17$ the authors of \cite{Tonoli} and
\cite{SchreyerTonoli}
discover three distinct families of Calabi--Yau threefolds.

Pfaffian Calabi--Yau threefolds, being one of the simplest Calabi--Yau threefolds
which are not
described as complete intersections in toric varieties, are good testing
examples
for the mirror symmetry conjecture.
The simplest of them, i.e. those which are arithmetically Cohen--Macaulay (or
equivalently in this case those of degree $d\leq 14$), have
already been studied with partial success from this
point of view (see \cite{Rodland,Bohm,Kanazawa}).
An interesting phenomenon occurs for those examples: the Picard-Fuchs equation
admits two points of maximal unipotent monodromy.
 Those points correspond to mirror partners with equivalent derived categories.
On the other hand, there is not a single representative in the huge database \cite{Database}
of Picard-Fuchs operators whose invariants would match the invariants of any
non-arithmetically Cohen-Macaulay Pfaffian Calabi--Yau threefold
(i.e. of degree $d\geq15$). Motivated by this, we have tried to understand the
geometry of these examples.

The construction of Tonoli families of degree 17 Calabi--Yau manifolds in
$\mathbb{P}^6$ can be
summarized as follows.
Let $W_3$, $P_7$ be two vector spaces of dimension 3 and 7 respectively. 
Then $\mathbb{P}(W_3\otimes P_7)$ contains a natural subvariety
$\operatorname{Seg}$ consisting of classes of simple tensors.
$\operatorname{Seg}$ is the image of the Segre embedding of
$\mathbb{P}(W_3)\times \mathbb{P}(P_7)$ into $\mathbb{P}(W_3\otimes P_7)$. In particular, we
have a map $\pi: \operatorname{Seg}\to \mathbb{P}(W_3)$.
For any point $\mathbf{P}$ in the Grassmannian $G(16,W_3\otimes P_7)$, we shall write $L_{\mathbf{P}}$
for the corresponding linear space of dimension $15$ in $\mathbb{P}(W_3\otimes P_7)$.
Consider
\begin{equation}\label{eqMk}\tilde{\mathcal{M}}_k=\{\mathbf{P}\in G(16,W_3\otimes
P_7)\mid \pi|_{L_{\mathbf{P}}\cap
\operatorname{Seg}} \text{ has $k$ distinct $\mathbb{P}^2$-fibers}\}.\end{equation}
Observe now that a general point $\mathbf{P}\in G(16,W_3\otimes P_7)$ defines a 
unique vector bundle of rank 13 on $\mathbb{P}(P_7)$ in the following way.
If $\mathbf{P}\in G(16,W_3\otimes P_7)$ then there is a natural map 
$L_{\mathbf{P}}\otimes P_7^*\to
W_3$ defined up to composition with an automorphism of $L_{\mathbf{P}}$ that 
defines a map
$$\lambda_{\mathbf{P}}:L_{\mathbf{P}}\otimes \mathcal{O}_{\mathbb{P}(P_7)} \to 
W_3\otimes
\mathcal{O}_{\mathbb{P}(P_7)}(1).$$
If this map is surjective (which happens for general $\mathbf{P}\in 
G(16,W_3\otimes
P_7)$), its kernel is a vector bundle that we shall call $E_{\mathbf{P}}$ or $E_{\lambda_{\mathbf{P}}}$.
Tonoli proves that, for each $k=8,9,11$, the family of bundles $E_{\mathbf{P}}$
parameterized by $\mathbf{P}$ in a non-empty open subset of
$\tilde{\mathcal{M}}_k$ defines a family of Pfaffian Calabi--Yau
threefolds of degree 17. For a formal construction of this family we refer to 
Section \ref{sec construction of famillies of Pfaffian}.

The main result of this paper is a simple construction of Tonoli
families of Calabi--Yau threefolds of degree $d=17$. The main ingredient of this
construction is the following theorem.
\begin{thm}\label{main result} For $k\in\{8,9,11\}$,
with the above notation, let $\mathcal{B}_k$ be the set of all $\mathbf{P}\in G(16,W_3\otimes P_7)$ 
such that:
\begin{enumerate}
 \item $k=11$ and  $L_{\mathbf{P}}$ contains the graph $\Gamma_{v_1}\subset
\operatorname{Seg}$ of a linear embedding $v_1:\mathbb{P}^2\to \mathbb{P}^6$;
 \item $k=9$ and $L_{\mathbf{P}}$ contains the graph $\Gamma_{v_{2}}\subset
\operatorname{Seg}$ of a second Veronese embedding  $v_2:\mathbb{P}^2\to
\mathbb{P}^6$;
 \item $k=8$ and $L_{\mathbf{P}}$ contains the closure of the graph 
$\Gamma_{v_{3}}$ of a
birational map $v_3: \mathbb{P}^2\to \mathbb{P}^6$ defined by a system
 of cubics passing through some point.
\end{enumerate}
Then $\mathcal{B}_k \cap \tilde{\mathcal{M}}_k$ contains an open and dense subset of both $\mathcal{B}_k$
and $\tilde{\mathcal{M}}_k$. 
\end{thm}
We also discuss in Remark \ref{M10} what happens in the case $k=10$; note that 
$\tilde{\mathcal{M}}_{10}\neq \emptyset$, but our construction
does not give rise to a codimension $3$ submanifold. The case $k\leq 7$ is 
discussed in Remark \ref{M7}; note that for $k=7$ 
the construction leads to a Gorenstein threefold that is not smooth. Moreover, 
for $k\leq 6$ we clearly have $\tilde{\mathcal{M}}_{k}\neq \emptyset$ but in this case again 
the construction does not lead to codimension 3 submanifolds.

Theorem \ref{main result} puts Tonoli's construction in a geometrical context
which makes it easier
to work with Tonoli examples. In particular, an explicit construction that
works
in characteristic $0$ (cf. \cite[\S 4]{Tonoli}) can be implemented in Macaulay
2 (see \cite{KKS}). Moreover, the structure of the moduli space of those 
examples can be described (the unirationality of the families becomes clear).
We also use our construction to recompute the dimensions of the Tonoli families
of Calabi--Yau threefolds in $\mathbb{P}^6$, and point out an error in Tonoli's
computation.
We prove that a Tonoli Calabi--Yau threefold of degree $d=17$ corresponding to
$k=11$
has Picard group of rank $\geq 2$.
\begin{cor} \label{rkPic =2}\it{There exists a Calabi--Yau threefold of degree $17$ 
in
$\mathbb{P}^6$ with Picard group of rank $\geq 2$.}
\end{cor}
 We conjecture that the relevant Tonoli family of Calabi--Yau threefolds of degree $d=17$ corresponding to
$k=11$ is locally complete, which implies in particular that the Hodge numbers of Calabi--Yau
threefolds
in this family are $h^{11}=2$ and $h^{12}=24$ (see Corollary \ref{rank Picard
group >=2}).

A second result of this paper is the description of del Pezzo surfaces in
$\mathbb{P}^5$ (embedded by a subsystem of the anti-canonical bundle) in terms
of Pfaffians of vector bundles and the observation of an analogy
 between these descriptions and descriptions of Tonoli
Calabi--Yau threefolds in $\mathbb{P}^6$. In the following tables we present
the vector bundles corresponding to these descriptions.
\begin{longtable}{|c|c|c|}
\hline
 Degree & Vector bundle defining projected del Pezzo surfaces in $\mathbb{P}^5$
\\
\hline
$3$& $\mathcal{O}_{\mathbb{P}^5}(-1)\oplus 2 \mathcal{O}_{\mathbb{P}^5}(1)$\\
\hline
$4$& $ 2 \mathcal{O}_{\mathbb{P}^5}\oplus \mathcal{O}_{\mathbb{P}^5} (1) $ \\
\hline
$5$& $5\mathcal{O}_{\mathbb{P}^5}$ \\
\hline
$6$& $\Omega^1_{\mathbb{P}^5}(1)\oplus 2 \mathcal{O}_{\mathbb{P}^5} $\\
\hline
$7$& $\operatorname{ker}(\psi)$, where $\psi\colon 11\mathcal{O}_{\mathbb{P}^5}
\to 2 \mathcal{O}_{\mathbb{P}^5}(1) $ is a general map\\
\hline
$8$&  $\operatorname{ker}(\psi)$, where $\psi\colon
14\mathcal{O}_{\mathbb{P}^5}\to 3 \mathcal{O}_{\mathbb{P}^5}(1) $ is a special
map with more syzygies \\
\hline
$9$&  $\operatorname{ker}(\psi)$, where $\psi\colon
17\mathcal{O}_{\mathbb{P}^5}\to 4 \mathcal{O}_{\mathbb{P}^5}(1) $ is a special
map with special  syzygies \\
\hline

\caption{Vector bundles defining del Pezzo surfaces}

\end{longtable}

\begin{longtable}{|c|c|c|}
\hline
 Degree & Vector bundle defining the Tonoli examples of Calabi--Yau
threefolds \\
\hline
$12$&  $\mathcal{O}_{\mathbb{P}^6}(-1)\oplus 2 \mathcal{O}_{\mathbb{P}^6}\oplus
2 \mathcal{O}_{\mathbb{P}^6} (1)$\\
\hline
$13$& $ 4 \mathcal{O}_{\mathbb{P}^6}\oplus \mathcal{O}_{\mathbb{P}^6} (1) $ \\
\hline
$14$& $7\mathcal{O}_{\mathbb{P}^6}$  \\
\hline
$15$& $\Omega^1_{\mathbb{P}^6}(1)\oplus 3 \mathcal{O}_{\mathbb{P}^6} $\\
\hline
$16$& $\operatorname{ker}(\psi)$, where $\psi\colon 13\mathcal{O}_{\mathbb{P}^6}
\to 2 \mathcal{O}_{\mathbb{P}^6}(1) $ is a general map\\
\hline
$17$&  $\operatorname{ker}(\psi)$, where $\psi\colon
16\mathcal{O}_{\mathbb{P}^6}\to 3 \mathcal{O}_{\mathbb{P}^6}(1) $ is a special
map with more syzygies \\
\hline

\caption{Vector bundles defining Tonoli Calabi--Yau threefolds}
\end{longtable}

Observe that a Tonoli Calabi--Yau threefold of degree 12 is just a complete
intersection of type $(2,2,3)$ and is naturally described by the Pfaffians
of $\mathcal{O}_{\mathbb{P}^6}\oplus 2 \mathcal{O}_{\mathbb{P}^6} (1)$. We have
changed this vector bundle to an equivalent one (see the proof of \cite[Lem.
3.4]{CYP6}) in
order to make the analogy more transparent.
Keeping this in mind, we associate to any del Pezzo surface $D\subset
\mathbb{P}^5$ a vector bundle $E_D$ from Table 1
defining it and in the same way to any Calabi--Yau threefold $X$ a bundle $F_X$
from Table 2 defining it. Observe that the bundle $E_D$ (resp. $F_X$) is
uniquely determined by
the degree $\deg(D)$ (resp. $\deg(X)$) when  $\deg(D)\leq 7$ (resp.
$\deg(X)\leq
16$). For del Pezzo surfaces of degree 8 and Tonoli
Calabi--Yau threefolds of degree 17, we in fact have $E_D=\textup{Syz}^1(HR(D))(-2)$ and
$F_X=\textup{Syz}^1(HR(X))(-3)$, where  $\textup{Syz}^1$ denotes the
sheafification of the first syzygy module of a given module, and $HR$ is the
Hartshorne--Rao module of a given variety (i.e. $HR(Y)=\bigoplus_{j\in
\mathbb{Z}} H^1(I_Y(j))$).

The analogy can now be formalized by the following theorem.
\begin{thm}\label{CY jako ext dP}
Let $X\subset \mathbb{P}^6$ be a general element of a family of Tonoli
Calabi--Yau threefolds and let $F_X$ be as above.
Then there exists a map $F_X\to 2O_{\mathbb{P}^6}$ whose kernel $E$
restricted to any $\mathbb{P}^5$ defines a del Pezzo surface of degree
$\deg(X)-9$. Conversely,
for
a general del Pezzo surface $D\subset \mathbb{P}^5$ of degree $\deg(D)\leq 8$
with associated bundle $E_D$
there exists an extension
$E'_D$ of the bundle $E_D$ to $\mathbb{P}^6$
such that a general bundle $F'_d$ fitting into a short exact sequence 
$$0\rightarrow E'_d \to F'_d \to 2 \mathcal{O}_{\mathbb{P}^6}\to 0 $$
defines a Calabi--Yau threefold in $\mathbb{P}^6$ of degree $\deg(D)+9$.
\end{thm}

This observation, rather straightforward for degree $\deg(X)\leq 16$ and 
$\deg(D)\leq 7$ (see Proposition \ref{analogy for d<7}), is
nontrivial for $\deg(X)=17$ and $\deg(D)=8$ (see Section \ref{sec rel dP Cy}).
Our proof in the latter case uses Theorem \ref{main result}.

Taking one step further, we conjecture an upper bound on
the
degree of Calabi--Yau threefolds in $\mathbb{P}^6$. More precisely, by analogy
to
the case of del Pezzo surfaces, we expect that there are
no smooth Calabi--Yau threefolds of degree $d\geq 19$ in $\mathbb{P}^6$.
Finally, we speculate about the possibility of constructing a degree
18 Calabi--Yau threefold
with description analogous to the one of a del Pezzo surface of degree 9.

The structure of the paper is the following. In Section \ref{sec Preliminaries},
we recall some basic facts from the theory of Pfaffians and provide some
preliminary results. In particular,
we describe a method to compute the dimensions of families of manifolds
obtained as Pfaffian varieties associated to families of vector bundles.
In Section \ref{sec construction of famillies of Pfaffian} we recall Tonoli constructions
in a slightly more general context and provide tools for the study of the resulting families. 
In particular we show how to compute the dimensions of these families.
In Section \ref{sec Del}, we quickly go through the constructions of del Pezzo
surfaces of degree $d_D\leq 7$ and Calabi--Yau threefolds of degree $d_X\leq
16$. We observe that
they are strictly related. In particular, Theorem \ref{CY jako ext dP} takes a
stronger form in these cases.
In Section \ref{sec Tonoli 17}, we provide various descriptions 
of the sets $\mathcal{M}_k$, compute their dimensions and 
prove Theorem \ref{main result}. We apply these results to the study of 
Tonoli families of Calabi--Yau threefolds of degree 17. In particular, we compute the
dimensions of these families. 

In Section \ref{dp8}, we describe anti-canonically embedded del Pezzo surfaces
of degree 8 in $\mathbb{P}^5$ in terms of Pfaffian varieties.
In Section \ref{sec rel dP Cy} we complete the discussion of the analogy between del Pezzo surfaces
and Tonoli Calabi--Yau threefolds and finish the proof of Theorem \ref{CY jako
ext dP}.

In Section \ref{final 9 18}, we make the conjecture that 18 is the
highest degree a canonical surface in
$\mathbb{P}^5$ can have. This would imply that the same bound holds for the
degree of Calabi--Yau threefolds in $\mathbb{P}^6$.

\subsection*{Acknowledgments} We would like to thank Ch.~Okonek for all his
advice
and support, and A.~Boralevi, S.~Cynk, D.~Faenzi, L.~Gruson, A.~Kresch,
A.~Langer,
P.~Pragacz, J.~Weyman for comments and discussions.
The use of Macaulay 2 was essential to guess the geometry.
The first author was supported by Iuventus Nr IP2011 005071 ``Uk\l{}ady linii na
zespolonych rozmaito\'{s}ciach kontaktowych oraz uog\'{o}lnienia'', the second by
MNSiW,
N N201 414539 and by the Forschungskredit of
the University of Zurich.

\section{Preliminaries}\label{sec Preliminaries} In this section, we recall some
basic facts of the theory of Pfaffians that will be needed in our construction.
Let $Y\subset \mathbb{P}^n$ be a Calabi--Yau
threefold in $\mathbb{P}^6$ or a del Pezzo surface in $\mathbb{P}^5$. Then, by
\cite{Walter}, the variety $Y$ can be described as a Pfaffian variety associated
to some vector bundle $E_Y$ of rank $2r+1$ on $\mathbb{P}^n$. Consider the 
related Pfaffian
resolution
\[ 0\to\mathcal{O}_{\mathbb{P}^n}(-2s-t)\to
E_Y^{\ast}(-s-t)\xrightarrow{\varphi}
E_Y(-s)\to\mathcal{I}_X\to 0, \] with $s=c_1(E_Y)+rt$.
Observe that by changing $E_Y$ to a suitable twist of it we may assume $t=1$.
Moreover,
by Formula (\ref{formula for canonical}) for the canonical class, $s=3$ for $Y$ being a Calabi--Yau
threefold and $s=2$ for
$Y$ being a del Pezzo surface.

By \cite{Walter} (or more precisely by an observation of Schreyer based on
\cite{Walter}),
under the assumption $h^i(\mathcal{O}_Y)=0$ for $0<i< \dim Y$, the bundle $E_Y$
is, up to a possible direct sum of line bundles, the
 twist by $\mathcal{O}_{\mathbb{P}^n}(-s)$ of the sheafification of the first
syzygy module of the Hartshorne--Rao module
$\bigoplus_{j\in \mathbb{Z}} H^1(I_Y(j))$.

In our constructions, we shall deal only with varieties
satisfying a series of additional assumptions, which are believed to be
satisfied by a general element of a Hilbert
scheme of subcanonical codimension $3$ varieties.
For this reason, in this preliminary section, we shall also make these
assumptions.
In particular, we shall assume that the submanifolds under study satisfy the
so-called maximal rank assumption stating that the restriction maps
$$H^0(\mathbb{P}^n,\mathcal{O}_{\mathbb{P}^n}(i))\to H^0(X,\mathcal{O}_X(i))$$
are of maximal rank for $i\in \mathbb{N}$. Moreover, the Hartshorne--Rao module
will be assumed to be generated in its smallest degree.
Its shift by $s$, which we shall usually denote by $M$ and call the shifted
Hartshorne--Rao module, will then be generated in degree $-1$.

These assumptions restrict our attention to varieties whose shifted
Hartshorne--Rao module $M$ has a presentation
\begin{equation} \label{Assumption presentation}
 p \mathfrak{S}_{\mathbb{P}^n} \to q\mathfrak{S}_{\mathbb{P}^n} (1)\to M\to 0,
 \end{equation}
 with $\mathfrak{S}_{\mathbb{P}^n} $ being the coordinate ring of
$\mathbb{P}^n$.
In this case, $M$ is determined by the map in the presentation, which itself is
given by a matrix of linear forms on $\mathbb{P}^6$. If, now, this matrix is
general enough, then
$M$ is determined up to isomorphism by the associated embedding
\begin{equation}\label{eq AS}\mathbb{P}^{p-1}\subset \langle
\mathbb{P}^6\times\mathbb{P}^{q-1} \rangle=\mathbb{P}^{7q-1}.\end{equation}
Note that we use the notation $\langle \mathbb{P}^6\times\mathbb{P}^{q-1}
\rangle$
for the codomain of the Segre embedding.
In the above context, the idea of our constructions will be to find a suitable
family of projective subspaces $\mathbb{P}^{p-1}\subset \langle
\mathbb{P}^6\times\mathbb{P}^{q-1}\rangle$ determining some family of modules.
The latter by the remark of Schreyer
should give rise to a family of vector bundles which, by further choices,
induces
a family of Pfaffian manifolds associated to these bundles. The first obstacle
in the above is that in the remark of Schreyer the direct sum of line bundles is
not determined
uniquely. This problem is solved by the following lemma, which together with
Formula (\ref{formula for canonical}) for the canonical class of a Pfaffian variety will give us the desired
uniqueness in the cases
studied.
\begin{lemm}\label{lem tylko dodatnie skladniki proste dla CY}
Let $M$ be the shifted Hartshorne--Rao module of a Calabi--Yau threefold $X$ of
degree $d$
in
$\mathbb{P}^6$.
Assume that $M$ has a presentation
$$ p \mathfrak{S}_{\mathbb{P}^6} \to q\mathfrak{S}_{\mathbb{P}^6} (1)\to M\to
0$$
with $p,q\in \mathbb{N}$ and $p-q\geq 3$.
Then $X$ is defined as a Pfaffian variety $\operatorname{Pf}(\sigma)$ for some $\sigma\in H^0((\bigwedge^2 E)(1))$, where
$E=Syz^1(M)\oplus \bigoplus_{i=1}^k \mathcal{O}_{\mathbb{P}^6}(a_i)$
for some $k$ and some $a_1,\dots,a_k\geq 0$.
\end{lemm}
\begin{proof}
Since, by definition, $M$ has only finitely many nonzero graded components,
the sheaf $Syz^1(M)$ is a vector bundle of rank $p-q$ obtained as the
kernel of the sheafified map in the presentation of $M$, i.e., we have
\begin{equation} \label{rezolwenta Syz^1(M)}
0\to \textup{Syz}^1(M) \to p \mathcal{O}_{\mathbb{P}^6} \to
q\mathcal{O}_{\mathbb{P}^6}(1) \to 0.
\end{equation}
We know that $E=\textup{Syz}^1(M)\oplus
\bigoplus_{i=1}^k \mathcal{O}_{\mathbb{P}^6}(a_i)$ for some $k$ and $a_i$. The
only thing we need to prove is that $a_i\geq 0$. Observe first that
from (\ref{rezolwenta Syz^1(M)}) we infer that $\textup{Syz}^1(M)$ has no
section and $c_1(\textup{Syz}^1(M))=-q$.
By analogous reasoning to \cite[Lemma 3.5]{CYP6} and by
Formula (\ref{formula for canonical}) for the canonical class we conclude that the number of negative $a_i$'s, if nonzero, must be smaller than $3-(p-q)$. The latter is negative by assumption so all $a_i\geq 0$.
\end{proof}
\begin{rem}\label{rem tylko dodatnie skladniki proste dla del pezzo}
A similar lemma is true for anti-canonically embedded del Pezzo surfaces in
$\mathbb{P}^5$ and $p-q\geq 2$.
\end{rem}

Since we are sometimes working in parallel with threefolds and with their surface
sections, it is good to keep in mind the following lemma relating the Hartshorne--Rao module
of a variety to the Hartshorne--Rao module of its hyperplane section:
\begin{lemm} Let $X\subset \mathbb{P}^n$ be a variety satisfying
$h^2(\mathcal{I}_{X|\mathbb{P}^n}(j))=0$ for any $j\in \mathbb{Z}$.
Let $M$ be the shifted Hartshorne--Rao module of $X$. Assume that $M$ has a
presentation
$$  p \mathfrak{S}_{\mathbb{P}^n}(-1)   \xrightarrow{m}  q
\mathfrak{S}_{\mathbb{P}^6}\to M\to 0,$$
given by a matrix $m$ of linear entries in the coordinate ring
$\mathfrak{S}_{\mathbb{P}^n}$ of $\mathbb{P}^n$.
Let $H$ be a hyperplane defined by a linear equation $h=0$.
Then the shifted Hartshorne--Rao module $M'$ of $X\cap H$ has a presentation
$$  p \mathfrak{S}_{H}(-1)   \xrightarrow{m'}  q \mathfrak{S}_{H}\to M'\to 0,$$
 with
$\mathfrak{S}_{H}=\mathfrak{S}_{\mathbb{P}^n}/_{\langle h\rangle}=\mathfrak{S}_{\mathbb{P}^{
n-1}}$ the coordinate ring of the hyperplane $H$,
 and $m'$ the image of $m$ via the projection map
$\mathfrak{S}_{\mathbb{P}^n}\to \mathfrak{S}_{\mathbb{P}^n}/_{\langle h\rangle}$.
\end{lemm}
\begin{proof}
 For each $j$, we have the exact sequence
 $$0\to I_{X|\mathbb{P}^n}(j-1)\xrightarrow{\lambda} I_{X|\mathbb{P}^n}(j) \to
I_{(X\cap H)|H}(j) \to 0,$$
 where $\lambda$ is given by multiplication by $h$.
 From the associated cohomology sequence in each degree and the assumed
vanishing $h^2(I_{X|\mathbb{P}^n}(j))=0$,  we obtain
 $M'={M}/{(hM)}$, and the presentation follows.
\end{proof}
\section{Constructions of families of Pfaffian varieties}\label{sec construction 
of famillies of Pfaffian}

In this section we shall outline the main constructions of families of
Pfaffian varieties used throughout the paper. In particular, we put in a formal 
context the Tonoli construction of families of Calabi--Yau threefolds.
We aim at constructing flat families $\mathbb{P}^n\times S \supset
\mathcal{X}\to S$ of
subcanonical submanifolds of some projective space $\mathbb{P}^n$.
\subsection{Families given by a fixed vector bundle}\label{family for fixed vb}
The first construction of a family of Pfaffian varieties that comes to mind is
to consider a vector bundle $E$ of rank $2r+1$ such that $h^0(\bigwedge^2
E(1))>0$
and consider all Pfaffian varieties of the form $\operatorname{Pf}(\sigma)$ with $\sigma\in
H^0(\bigwedge^2 E(1))$. More precisely, consider the open subset $U\subset
H^0(\bigwedge^2 E(1))$
defined by\[\textstyle U=\{\sigma \in H^0((\bigwedge^2 E)(1))\mid \operatorname{codim} (D^{2r-2} (\sigma))=3\}.\] 
The family of Pfaffian varieties parameterized by $U$ is defined as follows.
Let $\pi_1: U\times \mathbb{P}^n \to U$ and $\pi_2: U\times \mathbb{P}^n \to
\mathbb{P}^n$ be the natural projections. Consider the bundle $\pi_2^* \bigwedge^2 
E(1)$  and its evaluation
section $$\theta:U\times \mathbb{P}^n \ni (\sigma,x)\mapsto \sigma(x) \in \textstyle (\bigwedge^2 E)(1).$$
Then $\mathcal{X}=\operatorname{Pf}(\theta)\subset U\times \mathbb{P}^n$ is a codimension 3
subvariety and
$\pi_1|_{\mathcal{X}}:\mathcal{X}  \to U$ is a flat morphism.
Indeed, every fiber is a codimension 3 variety with the same Hilbert polynomial
computed by the Pfaffian resolution involving the same bundle.
In this way, we construct a family of Pfaffian varieties. This method is
sufficient for the construction of  locally complete families of Pfaffian
Calabi--Yau threefolds of degree $\leq 16$.
\begin{example}\label{Tonoli families of degree 16 as Pfaffian families}
Note that the Tonoli families of Calabi--Yau threefolds of degree $\leq 16$ are all 
obtained via the above construction using bundles from the table in the introduction.   
\end{example}

We shall now describe how to compute the dimensions of such
families of Pfaffian submanifolds, i.e. the dimension of the image $\mathcal{D}$
of the forgetful map
\[\phi\colon U \to \textup{Hilb}_{\mathcal{X}_{\sigma}|\mathbb{P}^n} ,\]
where  $\textup{Hilb}_{\mathcal{X}_{\sigma}|\mathbb{P}^n}$ is the Hilbert scheme containing 
$\pi_2(\pi_1|_{\mathcal{X}}^{-1}(\sigma))$ for chosen $\sigma\in U$.

\begin{prop}\label{Cor dimension of Pfaffians of one vector bundle} The
dimension
of the family of varieties obtained as Pfaffian varieties associated to a bundle
$E=\ker(p\mathcal{O}_{\mathbb{P}^n}\to q\mathcal{O}_{\mathbb{P}^n}(1)) \oplus
\bigoplus_{i=1}^k \mathcal{O}_{\mathbb{P}^6}(a_i)$  for $a_i\geq 0$ is
$h^0(\bigwedge^2 E(1))-\dim \mathrm{Hom}(E,E)$.
\end{prop}
Before proving Proposition \ref{Cor dimension of Pfaffians of one vector bundle} let us formulate a preparatory result.

\begin{lemm} \label{rezolwenta NX} Let $Y$ be a smooth variety obtained as a
Pfaffian variety associated to a bundle $E$ on $\mathbb{P}^n$. Then we have the
following exact sequence:
 \begin{equation}\label{eq rezolwenta NX}
0\to E^{\ast}(-s-t)\to E(-s)\oplus (S^2 E^{\ast})(-t)\to E\otimes E^{\ast}\to
(\bigwedge^2 E)(t)\to \mathcal{N}_{Y|\mathbb{P}^n}\to 0 ,  
 \end{equation}
where $\mathcal{N}_{Y|\mathbb{P}^n}$ is the normal bundle of $Y$ in
$\mathbb{P}^n$.
\end{lemm}
\begin{proof}
 First arguing as in \cite[Prop.~2.4]{KleppeMiroRoig} we deduce that
$\mathcal{N}_{Y|\mathbb{P}^n}=\mathcal{E}xt^1(\mathcal{I}_Y,\mathcal{I}
_Y)=\bigwedge^2\mathcal{I}_Y(2s+t)$.
Then from \cite{Weyman} we obtain the free resolution of the sheaf
$\bigwedge^2\mathcal{I}_{Y}$.
\end{proof}

\begin{proof}[Proof of Proposition \ref{Cor dimension of Pfaffians of one vector
bundle}]Let us keep the notation preceding Proposition
\ref{Cor dimension of Pfaffians of one vector bundle}. Let moreover $X=\pi_2(
\pi_1|_{\mathcal{X}}^{-1}(\sigma))$ for a fixed general $\sigma\in U$. Then the 
map
$H^0((\bigwedge^2 E)(1))\to H^0(\mathcal{N}_{X|\mathbb{P}^n})$ in Lemma
\ref{rezolwenta NX}
is interpreted as the tangent map to the forgetful map  $\phi$ at $\sigma$.
We want to prove that the dimension of the image of $\phi$ is
$h^0((\bigwedge^2 E)(t))-\dim \mathrm{Hom}(E,E)$. It is enough to prove
that
the rank of this tangent map at the general point $\sigma$
is  $h^0((\bigwedge^2 E)(t))-\dim \mathrm{Hom}(E,E)$.
Splitting  the long exact sequence (\ref{eq rezolwenta NX}) 
into short ones, we get
 $$0\to F\to (\textstyle \bigwedge^2 E)(t)\to \mathcal{N}_{X|\mathbb{P}^n}\to 0,$$
 $$0\to G\to E\otimes E^{\ast}\to F\to 0,$$
 $$0\to E^{\ast}(-s-t)\to E(-s)\oplus (S^2 E^{\ast})(-t)\to G\to 0,$$
for some bundles $F$, $G$ on $\mathbb{P}^n$.

 Moreover,  from the long cohomology sequences of the exact sequence
 \begin{equation}\label{seq resolution of E}
  0\to E\to p\mathcal{O}_{\mathbb{P}^n} \oplus \bigoplus_{i=1}^k
\mathcal{O}_{\mathbb{P}^6}(a_i)\to q\mathcal{O}_{\mathbb{P}^n}(1)\to 0,
 \end{equation}
its twists, twisted duals, and the resulting resolution of $(S^2 E^{\ast})(-t)$
obtained from \cite{Weyman}:
$$ 0\to { \binom{q}{2}} \mathcal{O}_{\mathbb{P}^n}(-t-2)\to
q\left(p\mathcal{O}_{\mathbb{P}^n} \oplus \bigoplus_{i=1}^k
\mathcal{O}_{\mathbb{P}^6}(-a_i)\right)(-t-1)\to S^2\left(p\mathcal{O}_{\mathbb{P}^n}
\oplus \bigoplus_{i=1}^k \mathcal{O}_{\mathbb{P}^6}(-a_i)\right)(-t)\to (S^2
E^{\ast})(-t)\to 0
$$
we deduce that $h^0(G)=h^1(G)=0$.
It follows that $h^0(F)=h^0(E\otimes E^{\ast})=\dim \mathrm{Hom}(E,E)$. Since 
$h^0(F)$ is the kernel
of the tangent map to $\phi$ at $s$, we deduce that the rank of this tangent map 
at $\sigma$ is
$h^0((\bigwedge^2 E)(t))-\dim \mathrm{Hom}(E,E)$, which gives the assertion.
\end{proof}

\subsection{Families of Pfaffians defined by a family of vector bundles}
For our purposes, in particular for the description of families of degree 17 
Calabi--Yau threefolds in $\mathbb{P}^6$
as well as del Pezzo surfaces of degree 8 in $\mathbb{P}^5$, we shall need a 
more general
construction than the one proposed in Subsection \ref{family for fixed vb}. In 
this construction the bundle defining the Pfaffian
varieties will be allowed to change.
We proceed as follows.

Let $\mathcal{E}$ be a vector bundle on $\mathbb{P}^n\times B$ for some smooth
affine variety $B$. Let us denote by $\pi:\mathbb{P}^n\times B \to \mathbb{P}^n$ the natural projection.
For all $\beta\in B$ denote the restricted bundle
$\mathcal{E}|_{\mathbb{P}^n\times \{\beta\}}$ by $\mathcal{E}_\beta$. Moreover, by abuse of notation, write 
$(\bigwedge^2\mathcal{E})(1)$ for $\bigwedge^2\mathcal{E}\otimes \pi^*\mathcal{O}_{\mathbb{P}^n}(1) $. Assume now that for
some open subset  $U\subset
H^0((\bigwedge^2\mathcal{E})(1))$, for all $\sigma \in U$ and all $\beta\in B$ the restriction
$\sigma|_{\mathbb{P}^n\times \{\beta\}}$
defines a Pfaffian variety (i.e. of codimension 3). Let $$\theta: U\times 
\mathbb{P}^n\times B\ni (\sigma,x,\beta) \mapsto \sigma(x,\beta) \in \textstyle (\bigwedge^2\mathcal{E})(1)$$
be the evaluation section of $\pi_{\mathbb{P}^n\times 
B}^*((\bigwedge^2\mathcal{E})(1))$ with $$\pi_{\mathbb{P}^n\times B}:
U\times \mathbb{P}^n\times B \to \mathbb{P}^n\times B$$ the natural projection 
and let
$$\mathcal{X}=\operatorname{Pf}(\theta)\subset U\times \mathbb{P}^n\times B$$ be its Pfaffian locus.
Finally, denote by $\pi_{U,B}$ the natural projection $U\times 
\mathbb{P}^n\times B \to U\times B$.

\begin{lemm}\label{plaska rodzina}
With the notation above, $\pi_{U,B}|_{\mathcal{X}}$ is a flat morphism.
\end{lemm}
\begin{proof}
The only thing we need to check is the equality of the Hilbert polynomials 
of each fiber.
That follows from the Pfaffian exact sequence
 and the fact that all $\mathcal{E}_{\beta}$ have the same Hilbert polynomial, since they are restrictions of
$\mathcal{E}$ which is flat over $\mathbb{P}^n$, being locally free over $\mathbb{P}^n\times B$.
\end{proof}

Let us now make use of Lemma \ref{plaska rodzina} in the context of the
paper.

\subsection{Tonoli construction} \label{general Tonoli construction} Consider 
vector spaces $V$, $W$ and $P$ of dimension $p$, $q$, $n+1$
respectively. We have
$$V\otimes W \otimes P =\textup{Hom} (V^{\vee}, W) \otimes P  .$$
It follows that each element $m \in  V\otimes W \otimes P$ induces a map
$$\varphi_m\colon V^\vee \times \mathcal{O}_{\mathbb{P}(P)}
\rightarrow W \times \mathcal{O}_{\mathbb{P}(P)}(1),$$
which globally gives
$$\varphi\colon V^\vee \times \mathcal{O}_{V\otimes W \otimes P\times
\mathbb{P}(P)}
\rightarrow W \times \pi^* (\mathcal{O}_{\mathbb{P}(P)}(1)),$$
where $\pi:V\otimes W \otimes P\times \mathbb{P}(P)\to \mathbb{P}(P)$ is the
natural projection.
Let $B\subset V\otimes W \otimes P$ be the open subset given by $B=\{m\in
V\otimes W \otimes P \mid \varphi_m \text{ is surjective} \}$. Then
$\mathcal{E}_B:=(\ker \varphi)|_{B\times \mathbb{P}(P)}$ is a vector bundle.

For any $k$ let  now $B_k\subset B$ be an algebraic subset of $B$ such that 
for each $b\in B$ we have
$$h^0(\textstyle \bigwedge^2
\mathcal{E}_B|_{\{b\}\times \mathbb{P}(P) }(1))=k.$$ 
Let$\mathcal{E}_k=\mathcal{E}|_{B_k\times \mathbb{P}(P)}$. Then by Grauert
semicontinuity there is   an open subset $B'_k\subset B_k$  such that the 
natural
map
$$\textstyle H^0(\bigwedge^2 \mathcal{E}_k|_{B'_k\times \mathbb{P}(P)}(1))\to
H^0(\bigwedge^2 \mathcal{E}_k|_{\{b\}\times \mathbb{P}(P)}(1))$$ is an
isomorphism for each $b\in B'_k$.

Finally, if we know that for some $b$ in $B'_k$ there exists a section  $\sigma\in
H^0(\bigwedge^2 \mathcal{E}_k|_{\{b\}\times \mathbb{P}(P)}(1))$
such that $\operatorname{Pf}(\sigma)$ is a smooth codimension 3 submanifold in $\mathbb{P}(P)$ then 
by further restricting ourselves to an open subset
$B''_k\subset B_k$ and to an open subset $U$ of $H^0(\bigwedge^2 
\mathcal{E}_k|_{B_k\times \mathbb{P}(P)}(1))$,  we may apply Lemma \ref{plaska 
rodzina}  giving rise to a flat family $\mathcal{T}_{(B_k,U,p,q,n)}$ of smooth
codimension 3 submanifolds in $\mathbb{P}^n$.
Note that in this way $\mathcal{T}_{(B_k,p,q,n)}$ is a smooth family over an 
open subset
$B''_k\times U \subset B_k\times H^0(\bigwedge^2 \mathcal{E}_k|_{B_k\times 
\mathbb{P}(P)}(1))$.

\begin{defi} A family $\mathcal{T}_{(B_k,p,q,n)}$ obtained as above will be 
called a Tonoli family of Pfaffian manifolds. Maximal Tonoli families for given 
$k,p,q,n$ will be denoted $\mathcal{T}_{k,p,q,n}$.
\end{defi}
\begin{example} \label{tonoli CY as tonoli pfaffian} The three families of 
Calabi--Yau threefolds of degree 17 in $\mathbb{P}^6$ constructed by Tonoli in 
\cite{Tonoli}
are examples of Tonoli families of Pfaffian manifolds of type $\mathcal{T}_{(k,16,3,6)}$. 
Indeed, we choose $V=V_{16}$, $W=W_3$, $P=P_7$
three vector spaces of dimensions indicated by the subscripts.
We observe that we have a rational map  $\Psi: V_{16}^*\otimes W_3\otimes P_7 
\to G(16,W_3\otimes P_7)$ and consider
$B_k=\Psi^{-1}(\tilde{\mathcal{M}}_k)$ where $\tilde{\mathcal{M}}_k$ is given by Equation 
(\ref{eqMk}). We then observe that the isomorphism class of the resulting $(\mathcal{E}_k)_{\beta}$
depends only on $\Psi(\beta)$. If $\Psi(\beta)=\mathbf{P}$ we shall denote  $(\mathcal{E}_k)_{\beta}$ by $E_\mathbf{P}$.
It is then proven in \cite{Tonoli} that $h^0(\bigwedge^2 E_{\mathbf{P}}(1))= k$ 
for general $\mathbf{P}\in \tilde{\mathcal{M}}_k$. The idea of the argument is as follows:
each special $\mathbb{P}^2$ fiber produces a section of $\bigwedge^2 E_{\mathbf{P}}(1)$;
we then check by a Macaulay 2 computation that these sections are independent and generate 
$h^0(\bigwedge^2 E_{\mathbf{P}}(1))$ for a specific randomly chosen $\mathbf{P}\in \mathcal{M}_k$
(see for example \cite{KKS} for methods to perform such a check) and conclude by semicontinuity.
By passing to open subsets we obtain a Tonoli family of Pfaffian manifolds 
$\mathcal{T}_{(\Psi^{-1}(\tilde{\mathcal{M}}_k),16,3,6)}$ which 
are the Calabi--Yau threefolds of degree 17 defined in \cite{Tonoli}.
\end{example}

We shall now present a of computing the dimension of such
Tonoli families of Pfaffian submanifolds, i.e. the dimension of the image 
$\mathcal{D}_{B_k,p,q,n}$
of the forgetful map
\[\varphi: B''_k \times U \to \textup{Hilb}_{X_{(b,\sigma)}|\mathbb{P}^6} ,\]
where $X_{(b,\sigma)}$ is the fiber of the family over the point $(b,\sigma)\in B''_k \times U$ and  
$\textup{Hilb}_{X_{b,\sigma}|\mathbb{P}^6}$ is the Hilbert scheme containing $X_{b,\sigma}$ for all
$(b,\sigma)\in B_k \times U$.

\begin{prop} \label{wymiar obrazu w schemat Hilberta}
Let  $\mathcal{T}_{(B_k,p,q,n)}$ be a Tonoli family of Pfaffian manifolds and 
let
$\mathcal{D}_{(B_k,p,q,n)}$ be the image of this
family under the forgetful map $\varphi$ to the Hilbert scheme as above. Then keeping the 
notation from the Tonoli construction above,
$$\dim \mathcal{D}_{(B_k,p,q,n)}= \dim B_k+k -p^2-q^2.$$
\end{prop}

\begin{proof}

Since the dimension of the domain of the forgetful map $\varphi$ is $\dim B_k+k$,
in order to compute the dimension of the image $\mathcal{D}_{(B_k,p,q,n)}$, we
need to compute the dimension of the fiber of $\varphi$.

Observe that in our case $E_{(b,\sigma)}=Syz^1(X_{(b,\sigma)})$ for $(b,\sigma)\in B''_k\times U$. 
It follows that if
$\varphi((b_1,\sigma_1))=\varphi((b_2, \sigma_2))$, then there exists an isomorphism
$\alpha: E_{b_1} \simeq E_{b_2}$. We know, moreover, that  every such isomorphism
$\alpha$ lifts to resolutions:

$$\begin{CD}0 @>>> E_{b_1} @>>> V_{16}\otimes \mathcal{O}_{\mathbb{P}(V_7)} @>>b_1>
V_{3}\otimes \mathcal{O}_{\mathbb{P}(V_7)} @>>> 0\\
@. @V\alpha VV @VAVV @VBVV  @.\\
0 @>>> E_{b_2} @>>> V_{16}\otimes \mathcal{O}_{\mathbb{P}(V_7)} @>>b_2> V_{3}\otimes
\mathcal{O}_{\mathbb{P}(V_7)} @>>> 0
\end{CD}$$

Now, from each fiber  $\varphi^{-1}(\varphi((b_1,\sigma_1))$ we have a map
$\pi_1:\varphi^{-1}(\varphi((b_1,\sigma_1))\to U$. By the above, the dimension of the
image of $\pi_1$ is equal to $\dim \operatorname{GL}(V)+\dim \operatorname{GL}(W)- 
\dim \operatorname{Aut} (E,E)$, whereas the
dimension of the fiber of $\pi_1$ is computed by Proposition \ref{Cor dimension
of Pfaffians of one vector bundle} to be $\dim \operatorname{Aut}(E,E)$.
It follows that the dimension of the fiber of $\varphi$ is $\dim \operatorname{GL}(V)+\dim
\operatorname{GL}(W)=p^2+q^2$, which ends the proof.
\end{proof}

\section{Del Pezzo surfaces of degree $\leq 7$ and Calabi--Yau threefolds of
degree $\leq 16$}\label{sec Del}
In this section we describe anti-canonically embedded del Pezzo surfaces of
degree $d\leq 7$ in
$\mathbb{P}^5$ in terms of Pfaffians of vector bundles.
Let us first make some general remarks on del Pezzo surfaces embedded in
$\mathbb{P}^5$ via a subsystem of the anti-canonical class.
\subsection{Del Pezzo surfaces in $\mathbb{P}^5$}
Recall that an anti-canonical model of a smooth del
Pezzo surface of degree $\geq 3$ is a smooth surface of degree $n$ in
$\mathbb{P}^n$ for $3\leq n\leq 9$. Moreover, for each degree $\neq 8$ there is
one family of such surfaces and for $n=8$ two families.

Consider anti-canonical embeddings of these surfaces in $\mathbb{P}^5$. More
precisely,
for $3 \leq n\leq 7$, we consider varieties obtained as the image of the anti-canonical embedding
of the del Pezzo surface of degree $n$ composed with a general linear map
$\mathbb{P}^n\to \mathbb{P}^5$. For $n=8$, we
have two del Pezzo surfaces $\mathbb{F}_1$ and $\mathbb{P}^1\times
\mathbb{P}^1$. So we have two types of del Pezzo surfaces of degree 8 in $\mathbb{P}^5$.

Let now $D$ be a del Pezzo surface of degree $d$ in $\mathbb{P}^5$ as above.
Then $D$ is clearly subcanonical, so by the theorem of Walter
it admits Pfaffian resolutions, which we shall study.

It follows from the Kodaira vanishing theorem and the Riemann--Roch theorem that
$H^i(\mathcal{I}_D)=0$ for $i>1$.
This implies that the bundle $E$ in the Pfaffian resolution of our del
Pezzo surface $D$  is the sheafification of
the module $Syz^1(\bigoplus_{k\in\mathbb{Z}}H^1(\mathcal{I}_D(k+2)))$ over the
coordinate ring of $\mathbb{P}^5$ plus a possible direct sum of line bundles.

\begin{lemm} \label{wielomian Hilberta del Pezzow} The Hilbert function of the
Hartshorne--Rao module of a del Pezzo surface $D\subset
\mathbb{P}^5$ of degree $d$ is $0$ for $d\leq 5$ and for $d\in \{6,7,8,9 \}$
takes the following values
starting from grade $0$: \\$(0,1,0,\dots)$, $(0,2,1,0,\dots)$,
$(0,3,4,0,\dots)$, $(0,4,7,0,\dots)$ respectively. Moreover, these del Pezzo
surfaces
satisfy the maximal rank assumption.
\end{lemm}
\begin{proof} We first check the maximal rank assumption by checking a random
example and concluding by semicontinuity as in
\cite[Lemma~5.1]{GKapustkaprimitive2}.
 The values of the Hilbert function are then computed from the Riemann--Roch
theorem as in \cite{Tonoli}.
\end{proof}
\subsection{Constructions of degree $\leq 7$ del Pezzo surfaces}
We can now get a description of a general del Pezzo surface of degree $d\leq 7$
in $\mathbb{P}^5.$
\begin{cor}
\it{A general del Pezzo surface of degree $d\leq 7$ in $\mathbb{P}^5$ is described
as a Pfaffian variety associated to the bundle:
\begin{enumerate}
\item  $\mathcal{O}_{\mathbb{P}^5}(-1)\oplus 2 \mathcal{O}_{\mathbb{P}^5}(1)$
for $d=3$,
\item  $2 \mathcal{O}_{\mathbb{P}^5}\oplus \mathcal{O}_{\mathbb{P}^5} (1)$ for
$d=4$,
\item  $5\mathcal{O}_{\mathbb{P}^5}$ for $d=5$,
\item  $\Omega^1_{\mathbb{P}^5}(1)\oplus 2 \mathcal{O}_{\mathbb{P}^5}$ for $d=6$,
\item  $\operatorname{ker}(\psi)$ for $d=7$, where $\psi\colon
11\mathcal{O}_{\mathbb{P}^5}\to 2 \mathcal{O}_{\mathbb{P}^5}(1) $ is a general
map. \\
\end{enumerate}}
\end{cor}
\begin{proof} From Lemma \ref{wielomian Hilberta del Pezzow} we know the bundles
up to a direct sum of line bundles. We next use the results of \cite[Section
3]{CYP6} and proceed analogously.
\end{proof}
\subsection{Analogy with Tonoli Calabi--Yau threefolds of degree $\leq 16$}
Recall that Tonoli families of Calabi--Yau threefolds of degree $k\leq 16$ are 
obtained by
the construction described in Section \ref{family for fixed vb}  applied to the
vector bundles in $\mathbb{P}^6$ characterized in Table 2.
Comparing the vector bundles appearing in the Pfaffian constructions of del
Pezzo surfaces and Tonoli Calabi--Yau
threefolds, we observe that the description of a general del Pezzo surface of
degree $d$ in $\mathbb{P}^5$ is similar
to the description of a general Tonoli Calabi--Yau threefold of degree $d+9$ in
$\mathbb{P}^6$. The relation is partially explained by the following.
\begin{prop}\label{degree and adjunction between del Pezzo and CY}
Let $E$ and $F$ be vector bundles on $\mathbb{P}^5$ and $\mathbb{P}^6$
respectively, related by the
exact sequence
$$0 \rightarrow E \rightarrow F|_{\mathbb{P}^5}\rightarrow 2
\mathcal{O}_{\mathbb{P}^5}\rightarrow 0.$$
Assume moreover that both bundles define smooth codimension $3$ varieties
$X\subset
\mathbb{P}^6$ and $D\subset \mathbb{P}^5$.
Then $X$ is a Calabi--Yau threefold of degree $d$ if and only if $D$ is a del
Pezzo surface of degree $d-9$.
\end{prop}
\begin{proof}
 This follows from Formula (\ref{formula for canonical}) for the canonical class and from the
following formula for the degree of a Pfaffian variety defined by a vector bundle in terms of Chern
classes of the vector bundle.

\begin{lemm}[see \cite{Okonek}] \label{lem degree formula}If $E$ is a vector
bundle of rank $2r+1$ on $\mathbb{P}^n$ and
$s\in H^0(\bigwedge^2 E(1))$ a general section that defines, via the
Pfaffian construction, a variety $Y$ of codimension 3. Then
\begin{multline*}
\deg(Y)= r c_1^2(E)H^{n-2}+c_1(E)c_2(E)H^{n-3}+(r^2+r)c_1(E)H^{n-1}+\\
c_2(E)H^{n-2}-c_3(E)H^{n-3}+ \frac{r(2r+1)(2r+2)}{12} H^n.
\end{multline*}
\end{lemm}
\begin{proof}
The proof is based on a computation using the Hirzebruch--Riemann--Roch theorem,
the restriction of the Pfaffian sequence to
a general $\mathbb{P}^3$, and the fact that the degree of a set of distinct
points is equal to the Euler characteristic of its structure sheaf. 
\end{proof} 
\end{proof}
Let us make the analogy more precise by proving Theorem \ref{CY jako ext dP} for
$d\leq 7$. Let $D$ be a del Pezzo surface of degree $d$ in $\mathbb{P}^5$, and
$E_D$ be the vector bundle on $\mathbb{P}^5$ defining $D$ through the Pfaffian
construction.
Consider a Tonoli Calabi--Yau threefold $X$ of degree $d+9$ and its associated
bundle $F_X$.

Observe that, for $d\leq 7$, the bundles $E_D$ and $F_X$ are determined by $d$
up
to a sum of rank 2 bundles of the form $\mathcal{O}(-i)\oplus
\mathcal{O}(i-1)$.
For our purpose, we choose the bundles from Tables 1 and 2 and denote them $E_d$
and $F_d$ respectively.

\begin{prop}\label{analogy for d<7}
 For $d\leq 7$, the bundle $E_d$
is obtained as the cokernel of a general surjective map
$F_d|_{\mathbb{P}^5}\rightarrow 2 \mathcal{O}_{\mathbb{P}^5}$.
Moreover, the bundle $E_d$ admits an
extension $E'_d$ to $\mathbb{P}^6$ such that t$F_d$ is a general bundle fitting into 
a short exact sequence
$$0\rightarrow E'_d \to F_d \to 2 \mathcal{O}_{\mathbb{P}^6}\to 0.$$ 
\end{prop}
\begin{proof}
 For each of the bundles $F_d$ for $d\leq 7$ we compute the restriction to a
general $\mathbb{P}^5$. We get
 \begin{itemize}
  \item $F_3|_{\mathbb{P}^5}=\mathcal{O}_{\mathbb{P}^5}(-1)\oplus 2
\mathcal{O}_{\mathbb{P}^5}\oplus  2 \mathcal{O}_{\mathbb{P}^5}(1),$
  \item $F_4|_{\mathbb{P}^5}=4\mathcal{O}_{\mathbb{P}^5}\oplus
\mathcal{O}_{\mathbb{P}^5} (1),$
  \item $F_5|_{\mathbb{P}^5}=7\mathcal{O}_{\mathbb{P}^5},$
  \item $F_6|_{\mathbb{P}^5}=\Omega^1_{\mathbb{P}^5}(1)\oplus 4
\mathcal{O}_{\mathbb{P}^5},$
  \item $F_7|_{\mathbb{P}^5}=2 \Omega^1_{\mathbb{P}^5}(1)\oplus
\mathcal{O}_{\mathbb{P}^5} \operatorname{ker}(\psi')$ for $\psi'\colon
13\mathcal{O}_{\mathbb{P}^5}\to 2 \mathcal{O}_{\mathbb{P}^5}(1) $  a general
map.
 \end{itemize}
 Note that $F_7$ as defined above is uniquely determined up to isomorphism.
It is now easy to check the first part of the proposition.
For the second part we take for $E'_d$ one of the following:
 \begin{itemize}
   \item $E'_3=\mathcal{O}_{\mathbb{P}^6}(-1)\oplus  2
\mathcal{O}_{\mathbb{P}^6}(1),$
  \item $E'_4=2\mathcal{O}_{\mathbb{P}^6}\oplus
\mathcal{O}_{\mathbb{P}^6} (1),$
  \item $E'_5=5\mathcal{O}_{\mathbb{P}^6},$
  \item $E'_6=\Omega^1_{\mathbb{P}^6}(1)\oplus
\mathcal{O}_{\mathbb{P}^6},$
  \item $E'_7=\operatorname{ker}(\psi'')$ for $\psi''\colon
11\mathcal{O}_{\mathbb{P}^6}\to 2 \mathcal{O}_{\mathbb{P}^6}(1) $ a general
map.
 \end{itemize}
It is clear that $E'_d|_{\mathbb{P}^5}=E_d$ for a general $\mathbb{P}^5\subset
\mathbb{P}^6$ and we conclude by observing that for each $d$ there is an exact
sequence
$$0\to E'_d\to F_d \to 2\mathcal{O}_{\mathbb{P}^6}\to 0,$$
and $F_d$ is always the general element fitting in the exact sequence. Indeed,
for $d\leq 6$ we have $\textup{Ext}^1 (2\mathcal{O}_{\mathcal{P}^6},E'_D)=0$ and
$F_d=E'_d\oplus 2\mathcal{O}_{\mathcal{P}^6}$, whereas for $d=7$
any bundle $F$ appearing in the exact sequence
$$0\to E'_7\to F \to 2\mathcal{O}_{\mathbb{P}^6}\to 0$$
is the kernel of some map $\theta\colon 13\mathcal{O}_{\mathbb{P}^6}\to 2
\mathcal{O}_{\mathbb{P}^6}(1)$, hence $F_d$ is general among them.
\end{proof}
\begin{rem}\label{wiakzi Fano i powierzchni gen typu}
The bundles $F_d|_{\mathbb{P}^5}$ and $E'_d$ for $d\leq 7$ define through the
Pfaffian construction general type surfaces in their canonical embedding in
$\mathbb{P}^5$ and del Pezzo threefolds in their half-anti-canonical embeddings
in
$\mathbb{P}^6$ respectively.
\end{rem}
\begin{cor}
Theorem \ref{CY jako ext dP} holds for del Pezzo surfaces of degree $ \leq
7$ and Tonoli Calabi--Yau threefolds of degree $\leq 16.$
\end{cor}
\begin{rem} In view of Propositions \ref{degree and adjunction between del Pezzo
and CY} and \ref{analogy for d<7}, it is natural to construct Calabi--Yau
threefolds and del Pezzo surfaces in pairs. In particular, having $E$ or $F$
one
can
try to reconstruct the other. The only thing missing and, in
fact, the most important thing from the point of view of the cases with pairs of 
degrees
$(8,17)$ and $(9,18)$ is the existence of a section
of the newly constructed $\bigwedge^2 F(1)$ and  $\bigwedge^2 E(1)$ defining a
smooth
Pfaffian variety (in particular of codimension 3).
\end{rem}
In the next sections, we shall study this phenomenon for del Pezzo surfaces of
degree
$8$ and Calabi--Yau threefolds of degree 17.

\section{Constructions of Tonoli revisited---degree $17$ Calabi--Yau
threefolds}\label{sec Tonoli 17}
In this section our aim is to find a geometric interpretation for Tonoli constructions of 
Calabi--Yau threefolds of degree $17$ in $\mathbb{P}^6$.
By Example \ref{tonoli CY as tonoli pfaffian}, these appear as Tonoli families of Pfaffian submanifolds 
of the form $\mathcal{T}_{B_k,6,16,3}$, where $B_k=\Psi^{-1}(\tilde{\mathcal{M}}_k)$ and 
$\Psi: V_{16}^*\otimes W_3\otimes P_7 \to G(16,W_3\otimes P_7)$ is the natural map associating to a matrix of 
linear forms the span of its columns. Let us keep this notation throughout the section. 

The reinterpretation of Tonoli constructions of Calabi--Yau threefolds of degree 17 in $\mathbb{P}^6$, 
claimed in the introduction, relies on finding good geometric constructions of the sets  
$\tilde{\mathcal{M}}_k \subset G(16,W_3\otimes P_7)$. 

Observe first that to any $\mathbf{P}\in G(16,W_3\otimes P_7)$ we can associate a unique
$\mathcal{G}_{\mathbf{P}}=\coker{\lambda_{\mathbf{P}}}$ on $\mathbb{P}(W_3)$,
where $\lambda_{\mathbf{P}}: 5\mathcal{O}_{\mathbb{P}(W_3)}(-1)\rightarrow
P_7\otimes \mathcal{O}_{\mathbb{P}(W_3)}$ is any embedding induced by the five linear 
equations defining $\mathbb{P}^{15}\subset\mathbb{P}(W_3\otimes W_7)$ (note that $\mathcal{G}_{\mathbf{P}}$ is independent of the choice of 
$\lambda_{\mathbf{P}}$).
Now, the condition $\mathbf{P}\in \tilde{\mathcal{M}}_k$ is equivalent to the condition
that the projectivization $\mathbb{P}(\mathcal{G}_{\mathbf{P}})$ has $k$ special fibers.

Now, defining $C_{\mathbf{P}}$ by the exact sequence
$$0\to \mathcal{G}_{\mathbf{P}}^{\vee\vee} \to \mathcal{G}_{\mathbf{P}}\to 
C_{\mathbf{P}}\to 0,$$
we can consider 
$$\mathcal{M}_k=\{\mathbf{P}\in G(16,W_3\otimes P_7) \mid \mathcal{C}_{\mathbf{P}} \ \text{is the 
structure sheaf of a scheme of length $k$} \}.$$

Since $\tilde{\mathcal{M}}_k\subset \mathcal{M}_k$ is a Zariski open subset, we shall 
look for good descriptions of $\mathcal{M}_k$.  

Let us first characterize each $\mathbf{P}\in \mathcal{M}_k$, for $k=8,9,11$,
by studying the sheaves $\mathcal{G}_{\mathbf{P}}$ corresponding to its 
elements.
\begin{prop} \label{wiazki w kojadrze dla CY stopnia 17}
If $\mathbf{P}\in G(16,W_3\otimes P_7)$ is a $\mathbb{P}^{15}$ in $\mathbb{P}(V_3\otimes
V_7)$
then $\mathbf{P}\in \mathcal{M}_k$ if and only if
$\mathcal{G}_{\mathbf{P}}^{\vee\vee}$ is a rank two vector bundle isomorphic to:
\begin{enumerate}
\item $T_{\mathbb{P}^2}(1)$,
\item $\mathcal{O}_{\mathbb{P}^2}(2)\oplus \mathcal{O}_{\mathbb{P}^2}(3)$,
\item $\mathcal{O}_{\mathbb{P}^2}(1)\oplus \mathcal{O}_{\mathbb{P}^2}(4)$,
\end{enumerate}
 for $k=8$, $9$ or $11$, respectively.
\end{prop}
\begin{proof} Let us first assume that $\mathcal{G}_{\mathbf{P}}^{\vee\vee}$ is 
one of the
above vector bundles for some $k\in\{8,9,11\}$.
Then we have exact sequence
$$  0\to 5\mathcal{O}_{\mathbb{P}^2}(-1)\xrightarrow{\lambda_k}
7\mathcal{O}_{\mathbb{P}^2}\to
\mathcal{G}_{\mathbf{P}}^{\vee\vee}\to C_{\mathbf{P}}\to 0.$$
By computing Chern classes  it follows that $\mathrm{rank}( \mathcal{C}_{\mathbf{P}})=0$, 
$c_1(\mathcal{C}_{\mathbf{P}})=0$ and $c_2(\mathcal{C}_{\mathbf{P}})=k$, hence $\mathcal{C}_{\mathbf{P}}$ is 
the
structure sheaf of a scheme of length $k$.

For the other direction let $\mathbf{P}\in \mathcal{M}_k$. Then we have an exact 
sequence
$$0\rightarrow 5\mathcal{O}_{\mathbb{P}^2}(-1) \rightarrow
7\mathcal{O}_{\mathbb{P}^2} \rightarrow 
\mathcal{G}_{\mathbf{P}}^{\vee\vee}\rightarrow C_{\mathbf{P}}
\rightarrow 0$$
with $C_{\mathbf{P}}$ the structure sheaf of a scheme of length $k$.
From that sequence we deduce that $c_1(\mathcal{G}_{\mathbf{P}}^{\vee\vee})=5H$, where $H$ is the class of a line in $\mathbb{P}^2$ and
$c_2(\mathcal{G}_{\mathbf{P}}^{\vee\vee})=(15-k) \mathrm{pt}$ where $\mathrm{pt}$ is the class of a point in $\mathbb{P}^2$.
Moreover, by the Bertini theorem, there exists a global section of
$\mathcal{G}_{\mathbf{P}}^{\vee\vee}$ vanishing in codimension 2.
This gives rise to an exact sequence
$$0\rightarrow \mathcal{O}_{\mathbb{P}^2}\rightarrow
\mathcal{G}_{\mathbf{P}}^{\vee\vee}\rightarrow I_{Z_{\mathbf{P}}}(5)\rightarrow 
0$$
where $Z_{\mathbf{P}}$ is a locally complete intersection scheme of length 
$15-k$ on
$\mathbb{P}^2$. It follows that $Z_{\mathbf{P}}$ satisfies
the Cayley--Bacharach property for quadrics, which means in our cases:
\begin{enumerate}
\item if $k=8$ then $Z_{\mathbf{P}}$ is a locally complete intersection scheme 
of length 7 with no
subscheme of length 6 contained in a conic;
\item if $k=9$ then $Z_{\mathbf{P}}$  is a locally complete intersection scheme 
of length 6 contained in
a conic;
\item if $k=11$ then $Z_{\mathbf{P}}$  is a locally complete intersection scheme 
of length 4 contained
in a line.
\end{enumerate}
We deduce that
\begin{itemize}
\item if $k=9$ then $Z_{\mathbf{P}}$ is a complete intersection of a quadric and 
a cubic, hence
$\mathcal{G}_{\mathbf{P}}^{\vee\vee}=\mathcal{O}_{\mathbb{P}^2}(2)\oplus
\mathcal{O}_{\mathbb{P}^2}(3)$;
\item if $k=11$ then $Z_{\mathbf{P}}$ is a complete intersection of a line and a 
quadric, hence
$\mathcal{G}_{\mathbf{P}}^{\vee\vee}=\mathcal{O}_{\mathbb{P}^2}(1)\oplus
\mathcal{O}_{\mathbb{P}^2}(4)$.
\end{itemize}
It remains to handle the case $k=8$. We have
$$0\rightarrow \mathcal{O}_{\mathbb{P}^2}\rightarrow
\mathcal{G}_{\mathbf{P}}^{\vee\vee}\rightarrow I_{Z_{\mathbf{P}}}(5)\rightarrow 
0.$$
Twisting by $\mathcal{O}_{\mathbb{P}^2}(-2)$ we obtain
$h^0(\mathcal{G}_{\mathbf{P}}^{\vee\vee}(-2))=3$ and 
$\mathcal{G}_{\mathbf{P}}^{\vee\vee}(-2)$ is
generated by these three sections up to a codimension 2 subset.
It follows that there is an exact sequence
\begin{equation}\label{exact sequence for G8}
0\rightarrow \mathcal{O}_{\mathbb{P}^2}\rightarrow
\mathcal{G}_{\mathbf{P}}^{\vee\vee}(-2)\rightarrow I_{p}(1)\rightarrow 0,
\end{equation}
with $p \in \mathbb{P}^2$.
This implies that for $k=8$ we have
$\mathcal{G}_{\mathbf{P}}^{\vee\vee}(-2)=T_{\mathbb{P}^2}(-1)$.
\end{proof}

\begin{cor}\label{if Mk then exist beta}
 In the notation of Proposition \ref{wiazki w kojadrze dla CY stopnia 17},
if
$\mathbf{P}\in \mathcal{M}_k$ then there exists a
 map $\beta_{\mathbf{P}}: 7\mathcal{O}_{\mathbb{P}^2}\to L_k$ surjective outside 
possibly a
set of
codimension at least 2, such that $\beta_{\mathbf{P}}\circ \lambda_{\mathbf{P}} 
=0$ and where $L_k$
 is one of the following sheaves:
 \begin{enumerate}
  \item $L_{11}=\mathcal{O}_{\mathbb{P}^2}(1)$,
  \item $L_{9}=\mathcal{O}_{\mathbb{P}^2}(2)$,
  \item $L_8=\mathcal{I}_{p}(3)$ for some $p\in\mathbb{P}^2$.
 \end{enumerate}
\end{cor}
\begin{proof}
 Let $\mathbf{P}\in \mathcal{M}_k$. Then by Proposition \ref{wiazki w kojadrze
dla CY stopnia 17} we have an exact sequence
 $$0\to 5\mathcal{O}_{\mathbb{P}^2}(-1)\xrightarrow{\lambda_{\mathbf{P}}}
7\mathcal{O}_{\mathbb{P}^2}\to
\mathcal{G}_{\mathbf{P}}^{\vee\vee}\to \mathcal{C}_{\mathbf{P}}\to 0$$
and
\begin{enumerate}
\item if $k=8$ then 
$\mathcal{G}_{\mathbf{P}}^{\vee\vee}=T_{\mathbb{P}^2}(1)$,
\item if $k=9$ then 
$\mathcal{G}_{\mathbf{P}}^{\vee\vee}=\mathcal{O}_{\mathbb{P}^2}(2)\oplus
\mathcal{O}_{\mathbb{P}^2}(3)$,
\item if $k=11$ then 
$\mathcal{G}_{\mathbf{P}}^{\vee\vee}=\mathcal{O}_{\mathbb{P}^2}(1)\oplus
\mathcal{O}_{\mathbb{P}^2}(4).$
\end{enumerate}
In each case we have a surjective map
$$\sigma_{\mathbf{P}}: \mathcal{G}_{\mathbf{P}}^{\vee\vee}\to L_k.$$
Indeed this is clear for $ k=9,11$, whereas for $k=8$ it follows from the exact
sequence (\ref{exact sequence for G8}).
Since the map $7\mathcal{O}_{\mathbb{P}^2}\to 
\mathcal{G}_{\mathbf{P}}^{\vee\vee}$ is
surjective outside codimension 2, so is its composition with 
$\sigma_{\mathbf{P}}$, giving rise to the desired $\beta_{\mathbf{P}}$.
\end{proof}
\begin{rem}
 Note that we have a map $\sigma_{\mathbf{P}}: T_{\mathbb{P}^2}(1) \to 
\mathcal{I}_p(3)$ for every $p\in \mathbb{P}^2$.
\end{rem}

\begin{cor} \label{construction via containing veronese k=9,11} Let $k\in 
\{9,11\}$. If there exists a surjective map $\beta_{\mathbf{P}}:
7\mathcal{O}_{\mathbb{P}^2}\to L_k$ such that $\beta_{\mathbf{P}}\circ 
\lambda_{\mathbf{P}} =0$, then
$\mathbf{P}\in \mathcal{M}_k$.

\end{cor}
\begin{proof}
Consider $\beta_{\mathbf{P}}: 7\mathcal{O}_{\mathbb{P}^2}\to L_k$ and 
$\lambda_{\mathbf{P}}: 5\mathcal{O}_{\mathbb{P}^2}(-1)
\to 7\mathcal{O}_{\mathbb{P}^2}$ such that $\beta_{\mathbf{P}}\circ 
\lambda_{\mathbf{P}}=0$. It follows that we have a surjective map
$\gamma_{\mathbf{P}}:\mathcal{G}_{\mathbf{P}}=\operatorname{coker} 
\lambda_{\mathbf{P}} \to L_k$. Now since $L_k$
is an injective sheaf for $k=9, 11$, we get in these cases a
surjective map $\gamma_{\mathbf{P}}': \mathcal{G}_{\mathbf{P}}^{\vee\vee}\to 
L_k$. Its kernel is then a line bundle with known Chern class, hence the line bundle
$$H_k=\begin{cases} \mathcal{O}_{\mathbb{P}^2}(4) \text{ for } k=11,\\
\mathcal{O}_{\mathbb{P}^2}(3) \text{ for } k=9.
\end{cases}$$

 \end{proof}

 \begin{cor}\label{construction via containing veronese k=8}
If there exists a surjective map 
$\beta_{\mathbf{P}}:7\mathcal{O}_{\mathbb{P}^2}\to \mathcal{I}_p(3)$ for some 
$p\in \mathbb{P}^2$ such that $\beta_{\mathbf{P}}\circ \lambda_{\mathbf{P}} =0$
 then we have two possibilities:
 \begin{enumerate}
  \item $\lambda_{\mathbf{P}}$ is degenerate at $p$ and $\mathbf{P}\in 
\mathcal{M}_9$,
  \item $\lambda_{\mathbf{P}}$ is not degenerate at $p$ and $\mathbf{P}\in 
\mathcal{M}_8$.
 \end{enumerate}
Moreover for a general map $\beta_{\mathbf{P}}:7\mathcal{O}_{\mathbb{P}^2}\to 
\mathcal{I}_p(3)$ a general  $\lambda_{\mathbf{P}}$ satisfying 
$\beta_{\mathbf{P}}\circ \lambda_{\mathbf{P}} =0$
is not degenerate at $p$. \end{cor}
\begin{proof}
 As in the proof of Corollary \ref{construction via containing veronese k=9,11},
 the map $\beta_{\mathbf{P}}$ induces a map $\gamma_{\mathbf{P}}: \mathcal{G}_{\mathbf{P}}=
 \operatorname{coker} \lambda_{\mathbf{P}} \to L_8=\mathcal{I}_p(3)$. Now, $\gamma_{\mathbf{P}}$
 extends to a map 
$\gamma_{\mathbf{P}}': \mathcal{G}_{\mathbf{P}}^{\vee\vee}\to \mathcal{O}_{\mathbb{P}^2}(3)$. The 
latter is either surjective or not. If it is surjective we have
$\mathcal{G}_{\mathbf{P}}^{\vee\vee}=\mathcal{O}_{\mathbb{P}^2}(2)\oplus
\mathcal{O}_{\mathbb{P}^2}(3)$, implying $\mathbf{P}\in \mathcal{M}_{9}$.
When $\gamma_{\mathbf{P}}'$ is not surjective then $\gamma_{\mathbf{P}}'$ maps onto 
$\mathcal{I}_p(3)$. Now, $\ker  \gamma_{\mathbf{P}}'$
is a line bundle with first Chern class of degree 2, thus $\mathcal{O}_{\mathbb{P}^2}(2)$.   
 It follows that $\mathcal{G}_{\mathbf{P}}^{\vee\vee}=T_{\mathbb{P}^2}(1)$, which in its turn implies 
by Proposition \ref{wiazki w kojadrze dla CY stopnia 17} that $\mathbf{P}\in \mathcal{M}_8$.
Finally, $\gamma'_{\mathbf{P}}$ is surjective if and only if $p$ is in the support 
of the sheaf $\mathcal{C}_\mathbf{P}=\coker (\mathcal{G}_{\mathbf{P}} \to \mathcal{G}_{\mathbf{P}}^{\vee\vee})$.
The latter is equivalent to $\lambda_{\mathbf{P}}$ being degenerate at $p$.
 \end{proof}
 \begin{rem} \label{M7}
  In order to characterize geometrically $\mathcal{M}_k$ for $k\leq 7$ one can 
use the same approach as above.
  For example if $k=7$ the exact sequence
  $$0\rightarrow 5\mathcal{O}_{\mathbb{P}^2}(-1) \xrightarrow{\lambda_7} 
7\mathcal{O}_{\mathbb{P}^2} \rightarrow \mathcal{G}_7^{\vee\vee}\rightarrow C_7
\rightarrow 0$$
shows that $\mathcal{G}_7^{\vee\vee}$ is a rank 2 vector bundle with 
$c_1(\mathcal{G}_7^{\vee\vee})=-5H$ and $c_2(\mathcal{G}_7^{\vee\vee})=8$ such 
that
$\mathcal{G}_7^{\vee\vee}$ admits a section vanishing in a 0-dimensional scheme 
$Z_7$ of length 8 satisfying the Cayley-Bacharach property for quadrics,
i.e. no subscheme $Z_7$ of ength 7 lies on a conic. We get the exact 
sequence
$$0\rightarrow \mathcal{O}_{\mathbb{P}^2}\rightarrow
\mathcal{G}^{\vee\vee}\rightarrow I_{Z_7}(5)\rightarrow 0$$
and deduce that $\mathcal{G}^{\vee\vee}(-2)$ has a 3-dimensional space of 
sections whose general element vanishes in codimension 2.
This gives the sequence
$$0\rightarrow \mathcal{O}_{\mathbb{P}^2}\rightarrow
\mathcal{G}^{\vee\vee}(-2)\rightarrow I_{K_2}(1)\rightarrow 0$$
for some scheme $K_2$ of length 2. After tensoring by 
$\mathcal{O}_{\mathbb{P}^2}(2)$ we get
$$0\rightarrow \mathcal{O}_{\mathbb{P}^2}(2)\rightarrow
\mathcal{G}^{\vee\vee}\rightarrow I_{K_2}(3)\rightarrow 0.$$
The surjection in that sequence is used to obtain the map $\beta_7: 
7\mathcal{O}_{\mathbb{P}^2}\rightarrow I_{K_2}(3)$
such that $\beta_7\circ \lambda_7=0$. Everything being general, the map $\beta_7$ 
will be a surjection.
For the inverse direction starting from a surjection
$\beta_7: 7\mathcal{O}_{\mathbb{P}^2}\rightarrow I_{K_2}(3)$ a general map 
$\lambda_7: 5\mathcal{O}_{\mathbb{P}^2}(-1) \to 7\mathcal{O}_{\mathbb{P}^2}$
satisfying $\beta_7\circ \lambda_7=0$ will correspond to an element in 
$\mathcal{M}_7$.
 \end{rem}
\begin{rem}\label{M10}
 Note that $\mathcal{M}_{10}\neq \emptyset$. In fact, elements of 
$\mathcal{M}_{10}$ correspond to maps
 $\lambda:5\mathcal{O}_{\mathbb{P}^2}(-1)\to 7 \mathcal{O}_{\mathbb{P}^2}$ such 
that there exists a commutative diagram

$$ \begin{CD} 5\mathcal{O}_{\mathbb{P}^2}(-1) @>\lambda>> 7 
\mathcal{O}_{\mathbb{P}^2}\\
@AAA @VVV\\
4\mathcal{O}_{\mathbb{P}^2(-1)} @>0>> 2 \mathcal{O}_{\mathbb{P}^2}
 \end{CD}$$
with vertical arrows being embeddings. Indeed, the degeneracy locus of a general 
such $\lambda$ is equal to the degeneracy locus
of the restricted map $4\mathcal{O}_{\mathbb{P}^2}(-1) \to 5 
\mathcal{O}_{\mathbb{P}^2}$ which is  a scheme of codimension
2 and degree 10. To prove that general elements of $\mathcal{M}_{10}$ arise in 
this way we proceed similarly to the cases $k=7,8,9,11$.
i.e. if $\mathbf{P}\in \mathcal{M}_{10}$ then we have an exact sequence
$$0\to 5\mathcal{O}_{\mathbb{P}^2}(-1)\to 7 \mathcal{O}_{\mathbb{P}^2}\to 
\mathcal{G}^{\vee\vee}\to \mathcal{C}_{10}\to 0$$
and it follows that a general section of $\mathcal{G}^{\vee\vee}$ vanishes in a 
subscheme $Z_5$ of dimension 0 and length 5 satisfying the Cayley--Bacharach
 property, hence contained in a line. We hence have an exact sequence
 $$0\to \mathcal{O}_{\mathbb{P}^2} \to \mathcal{G}^{\vee\vee}\to 
\mathcal{I}_{Z_5}(5)\to 0.$$
 But $\mathcal{I}_{Z_5}(5)$ has the following resolution:
$$0\to \mathcal{O}_{\mathbb{P}^2}(-1)\to \mathcal{O}_{\mathbb{P}^2}\oplus 
\mathcal{O}_{\mathbb{P}^2}(4) \to \mathcal{I}_{Z_5}(5)\to 0.$$
The map $\mathcal{G}^{\vee\vee}\to \mathcal{I}_{Z_5}(5)$ inducing a map 
$\mathcal{G}\to \mathcal{I}_{Z_5}(5)$ gives rise to the following diagram:

$$\begin{CD}0@>>> 5\mathcal{O}_{\mathbb{P}^2}(-1) @>\lambda>> 7 
\mathcal{O}_{\mathbb{P}^2} @>>> \mathcal{G} @>>> 0\\
@. @VVV @VVV @VVV @.\\
0@>>> \mathcal{O}_{\mathbb{P}^2(-1)} @>>> \mathcal{O}_{\mathbb{P}^2}\oplus 
\mathcal{O}_{\mathbb{P}^2}(4) @>>> \mathcal{I}_{Z_5}(5) @>>> 0
 \end{CD}$$
 This induces
$$ \begin{CD} 5\mathcal{O}_{\mathbb{P}^2}(-1) @>\lambda>> 7 
\mathcal{O}_{\mathbb{P}^2}\\
@AAA @VVV\\
4\mathcal{O}_{\mathbb{P}^2(-1)} @>0>> \mathcal{O}_{\mathbb{P}^2}
 \end{CD}$$
Now, in case $\lambda$ is general the degeneracy locus of $\lambda$ is equal to 
the degeneracy locus of the restricted map
$$4\mathcal{O}_{\mathbb{P}^2}(-1) \to 6 \mathcal{O}_{\mathbb{P}^2}$$
the latter by an analogous argument  is a scheme of dimension 0 and length 
10 only if it factorizes through a map
$$4\mathcal{O}_{\mathbb{P}^2}(-1) \to 5 \mathcal{O}_{\mathbb{P}^2},$$ which 
implies our general form of elements of $\mathcal{M}_{10}$.
\end{rem}
\begin{rem}
 Observe that, a priori, for each $k$ the set $\mathcal{M}_k$ may have several components.
Theorem \ref{main result}  concerns the components for which the map $\beta_{\mathbf{P}}$
is general. 
 \end{rem}
 \begin{rem}
 For $k=11$ each element $\mathbf{P}\in \mathcal{M}_k$ is of one of two types. One type has
 $\beta_{\mathbf{P}}\colon 7 \mathcal{O}_{\mathbb{P}^2}\to  \mathcal{O}_{\mathbb{P}^2}(1)$  surjective.
 The second type has $\beta_{\mathbf{P}}$ factorizing through the 
ideal of a point.
 In the latter case $\ker \beta_{\mathbf{P}}=5\mathcal{O}_{\mathbb{P}^2}\oplus 
\mathcal{O}_{\mathbb{P}^2}(-1)$. Since $\lambda_{\mathbf{P}}$ factorizes through
 the embedding $\ker \beta_{\mathbf{P}}\to 7\mathcal{O}_{\mathbb{P}^2} $, it must  
decompose as $\lambda_{\mathbf{P}}=\lambda_1\oplus \lambda_2$ with
 $\lambda_1:4 \mathcal{O}_{\mathbb{P}^2}(-1)\to  5\mathcal{O}_{\mathbb{P}^2}$ 
and
 $\lambda_2: \mathcal{O}_{\mathbb{P}^2}(-1) \to 2\mathcal{O}_{\mathbb{P}^2}$. 
Such maps $\lambda_{\mathbf{P}}$ do not however give rise to
 families of Calabi--Yau threefolds because the degeneracy locus of no skew-symmetric map
 $E_{\lambda_{\mathbf{P}}}^{\vee}(-1)\to E_{\lambda_{\mathbf{P}}}$ is of expected 
codimension.
\end{rem}
\begin{rem} \label{skew Calabi--Yau} In the family of Calabi--Yau threefolds of 
degree 17 with $k=9$ we can identify a subfamily obtained
in the following way. Consider
$$\varphi: 7 \mathcal{O}_{\mathbb{P}^2}\to \mathcal{O}_{\mathbb{P}^2}(3)$$
defined by Pfaffians of a skew-symmetric matrix. The syzygies of this map
recover the skew symmetric map
$$\theta: 7\mathcal{O}_{\mathbb{P}^2}(-1)\to 7\mathcal{O}_{\mathbb{P}^2}.$$

Considering the general map $$\iota: 5\mathcal{O}_{\mathbb{P}^2}(-1)\to
\mathcal{O}_{\mathbb{P}^2}(-2)\oplus 7 \mathcal{O}_{\mathbb{P}^2}(-1)$$
we find that
$$\theta\circ \iota: 5\mathcal{O}_{\mathbb{P}^2}(-1)\to 7
\mathcal{O}_{\mathbb{P}^2}$$
defines a Calabi--Yau threefold of degree $17$ with $k=9$. However, one can check
that in this way we can get only special Calabi--Yau threefolds with $k=9$.
\end{rem}

 \subsection{Proof of Theorem \ref{main result}}
 Let $\mathbf{P} \in \mathcal{B}_k \subset G(16,V_3\otimes V_7)$, i.e. one of the following holds: 
 \begin{enumerate}
 \item $k=11$ and $L_{\mathbf{P}}$ contains the graph $\Gamma_{v_1}\subset
\operatorname{Seg}$ of a linear embedding $v_1:\mathbb{P}^2\to \mathbb{P}^6$;
 \item $k=9$ and $L_{\mathbf{P}}$ contains the graph $\Gamma_{v_{2}}\subset
\operatorname{Seg}$ of a second Veronese embedding  $v_2:\mathbb{P}^2\to
\mathbb{P}^6$;
 \item $k=8$ and $L_{\mathbf{P}}$ contains the closure of the graph 
$\Gamma_{v_{3}}$ of a
birational map $v_3: \mathbb{P}^2\to \mathbb{P}^6$ defined by a system
 of cubics passing through some point.
\end{enumerate}
Then clearly the assumptions of Corollaries \ref{construction via containing veronese 
k=8} and \ref{construction via containing veronese k=9,11} also hold.
Hence, $\mathbf{P}\in \mathcal{M}_k$ by these corollaries. Moreover Proposition \ref{if Mk then exist beta} itmplies that $\mathcal{B}_k$ is Zariski open in 
$\mathcal{M}_k$.
Using Macaulay2 (with \cite{KKS}), in each case, among all $\mathbf{P} \in \mathcal{B}_k $ 
we can find elements of $\tilde{\mathcal{M}}_k$. We conclude that $\mathcal{B}_k\cap \mathcal{M}_k$ 
is a nonempty open subset of  $\mathcal{M}_k$. \qed

\subsection{Another construction}\label{subsec another constr in Cy}
Note that we also have an alternative way to describe $\mathbf{P}\in \mathcal{M}_k$.
For $k=8,9,11$, consider the projectivization of the bundle
$\mathcal{T}_{\mathbb{P}^2}(1)$,
$\mathcal{O}_{\mathbb{P}^2}(2)\oplus \mathcal{O}_{\mathbb{P}^2}(3)$ or
$\mathcal{O}_{\mathbb{P}^2}(1)\oplus \mathcal{O}_{\mathbb{P}^2}(4)$
respectively. The bundle is embedded in $\mathbb{P}^2\times \mathbb{P}^l\subset
\mathbb{P}^{3l+2}$ with $l=14, 15, 17$ respectively.
For each of these cases, consider the projection $\mathbb{P}^2\times
\mathbb{P}^l\to \mathbb{P}^2\times \mathbb{P}^6$ from the spaces spanned by
$\mathbb{P}^2 \times \mathbb{P}^{k-1}$,
where $\mathbb{P}^{k-1}$ is spanned by $k$ general points on the image
of the projectivization of the bundle on $\mathbb{P}^l$.
We obtain in this way a $\mathbb{P}^{15} \subset \mathbb{P}^{20}$ containing
$k$ special fibers, which are the proper transforms of the points from which we
projected.
From Proposition \ref{wiazki w kojadrze dla CY stopnia 17} a $\mathbb{P}^{15}$
in $\mathcal{M}_k$ must appear as the span of
the projection of the
corresponding bundle. The reason why the projection must be performed from
points
lying on the projectivization is the appearance of special fibers.
\subsection{Dimensions of $\mathcal{M}_k$ (cf. \cite[Prop. 3.5]{Tonoli})}  For each $k=8,9,11$ we compute the dimension of $\mathcal{B}_k$ from Theorem \ref{main result},
which is an open subset of an irreducible closed subvariety 
of $G(16, W_3\otimes P_7)$.
\begin{lemm}\label{wymiar M_k}
 The dimension of the space $\mathcal{B}_k$ for $k=8$, $9$, $11$ is given by:
 $$\dim \mathcal{B}_k=\begin{cases}
			72 \text{ for } k=8,\\
                       71 \text{ for } k=9,\\
                       70 \text{ for } k=11.
                      \end{cases}
$$ 
\end{lemm}
\begin{proof}
We consider each case separately:
\begin{enumerate}
\item The dimension of the space parameterizing graphs of linear embeddings $\mathbb{P}^2 \to \mathbb{P}^6$ is $20$.
The graph of each linear embedding spans a $\mathbb{P}^5$. The Grassmannian
parameterizing all $\mathbb{P}^{15}$
containing a fixed $\mathbb{P}^5$ is of dimension $50$. 
Since a general $\mathbb{P} \in \mathcal{B}_{11}$ contains only one 
graph of a linear map, the dimension of $\mathcal{B}_{11}$ is $70$.

\item The dimension of the space of quadratic embeddings $\mathbb{P}^2\to \mathbb{P}^6$ is 
$41$. The graph of each such embedding spans a $\mathbb{P}^{9}$. 
Hence the dimension of the family of $\mathbb{P}^{15}$ containing 
the graph of a fixed Veronese embedding is $30$.  Now, since a general
$\mathbb{P}\in \mathcal{B}_9 $  contains only one graph of a Veronese 
embedding, the dimension of $\mathcal{B}_9$ is $71$.

\item For a chosen point $p\in \mathbb{P}^2$ we have a $62$-dimensional family 
of graphs. A general graph spans a
$\mathbb{P}^{13}$, hence we have a $10$-dimensional family of $\mathbb{P}^{15}$'s 
for each
of them. Since for a fixed $p\in \mathbb{P}^2$ there is a unique graph, we conclude that
$\mathcal{B}_8$ is of dimension $72$.
\end{enumerate}
 
\end{proof}

\subsection{Dimension of Tonoli families}
The aim of this subsection is to estimate the Hodge numbers of Tonoli 
Calabi--Yau threefolds of degree 17 by proving
the following theorem:
\begin{thm}\label{dim-rodziny}
Let $X_{17}^k$ be a Tonoli Calabi--Yau threefold  defined as a general Pfaffian variety
associated to a vector bundle
$E_{\mathbf{P}}$ for some general $\mathbf{P}\in \mathcal{B}_k$, where $\mathcal{B}_k$ 
is as in Theorem \ref{main result} and $k\in \{8,9,11\}$. Then
$$h^{1,2}(X_{17}^k)\geq \dim \mathcal{B}_k-57+ k.$$ 
\end{thm}
\begin{rem} We have not been able to compute the exact value of $h^{1,2}$, but we expect
that we have equality in the inequality above.
\end{rem}
We shall need some preliminary results.
Let $X\subset \PP^6$ be a Calabi--Yau threefold. Denote by $\mathcal{H}_{X}$ the
component
of the Hilbert scheme $Hilb_{X|\PP^6}$ containing $X$. The tangent space to 
$\mathcal{H}_{X}$ at $X$ is naturally identified to
$H^0(N_{X|\PP^6})$.
Consider the following map locally around $X\in \mathcal{H}_X$:
\[\mathcal{H}_X\xrightarrow{\pi} Def(X),\] where $Def(X)$ is the local 
deformation space
with tangent space $H^1(T_X)$.

\begin{lemm}\label{QWE} Let $X\subset \PP^6$ be a Calabi--Yau threefold. Then
the natural map of tangent spaces $\tau\colon H^0(N_{X|\PP^6})\to H^1(T_X)$ is a
surjection.
In particular the deformations of $X$ can be embedded into $\PP^6$.
\end{lemm}
\begin{proof} The statement follows from the long exact cohomology sequence
constructed from
$$0\to T_X\to T_{\PP^6}|_{X}\to N_{X|\PP^6}\to 0$$ and from the Euler sequence
$$0\to \oo_X \to \oo_X(1)\to T_{\PP^6}|_X \to 0.$$
We also infer that $H^0(T_{\PP^6}|_X)=H^0(T_{\PP^6})$ is the kernel of $\tau$.
\end{proof}
Let us now introduce some notation fitting with the notation in Section 
\ref{sec construction of famillies of Pfaffian}. If $X$ is a Tonoli Calabi--Yau threefold
we shall denote by $\mathcal{T}_X$ the component containing $X$ of the appropriate Tonoli family 
of Pfaffian varieties, and $\mathcal{D}_{X}$ will stand for the image of the Tonoli family 
$\mathcal{T}_X$ under the forgetful map to $Hilb_{X|\mathbb{P}^n}$. In particular 
$\mathcal{D}_X\subset \mathcal{H}_X$.
\begin{prop}\label{48} Let $X\subset \PP^6$ be a Tonoli Calabi--Yau threefold.
Then $h^{1,2}(X)\geq \dim \mathcal{H}_X-48 \geq \dim \mathcal{D}_X-48$.
\end{prop}
\begin{proof} 
It follows from the proof of Lemma \ref{QWE} that $h^{1,2}(X)=\dim
Def(X)\geq \dim \mathcal{H}_{X}-h^0(T_{\PP^6})\geq \dim \mathcal{D}_X-48$.
\end{proof}
We can now prove Theorem \ref{dim-rodziny}.
\begin{proof}[Proof of Theorem \ref{dim-rodziny}]
The theorem is a direct consequence of Propositions \ref{48} and \ref{wymiar 
obrazu w schemat Hilberta}.

\end{proof}

\begin{cor}
 The deformation families of the constructed Tonoli Calabi--Yau threefolds are of 
dimensions $\geq$  $23$, $23$, $24$ for $k=8,9, 11$ respectively.
\end{cor}
\begin{proof}
 We apply Theorem \ref{dim-rodziny} and Lemma \ref{wymiar M_k}.
\end{proof}
\begin{cor}\label{rank Picard group >=2}
The rank of the Picard group of the family of Tonoli Calabi--Yau threefolds of
degree $17$ with $k=11$ is not smaller than $2$.
\end{cor}
\begin{proof}
Since the degree of each Calabi--Yau threefold $X$ in the family is 17, by the 
double point formula we get
$2(h^{1,1}(X)-h^{1,2}(X))=-44$. But by Theorem \ref{dim-rodziny} we have $h^{1,2}(X)\geq 
24$. Thus  $h^{1,1}(X)\geq 2$.
\end{proof}

\begin{rem} We believe that the Picard number of Tonoli Calabi--Yau threefolds 
of
degree $17$ with $k=11$ is in fact equal to $2$ and that in the
cases with $k=9,8$ it is equal to $1$, but we cannot prove it at the moment.
 It would be interesting to study the example with $k=11$ from the point of view 
of
rationality of the rays of the
K\"{a}hler cone as in \cite{LazicPeternell}.
\end{rem}
\begin{rem} The Tonoli Calabi--Yau threefold of degree 17 with $k=11$
constructed above shows that
the Barth-Lefschetz theorem cannot be generalized to subcanonical
threefolds in $\mathbb{P}^6$.
Another example of this phenomenon is the del Pezzo threefold of degree 7
projected to $\mathbb{P}^6$. It is obtained as the projection to $\mathbb{P}^6$
of
the second Veronese embedding of $\mathbb{P}^3$ from a $\mathbb{P}^2$
intersecting it in one point.

\end{rem}

\section{Descriptions of del Pezzo surfaces of degree $8$ in
$\mathbb{P}^5$}\label{dp8}
We shall now describe the Pfaffian resolutions of del Pezzo surfaces of
degree $d=8$. Our approach will be parallel to the case of Tonoli Calabi--Yau 
threefolds.
We shall look for modules $M$ with Hilbert function $(3,4,0,\dots)$ for
gradation starting from $-1$ which
admit a minimal resolution
$$14 \mathfrak{S}_{\mathbb{P}^5}\to 3\mathfrak{S}_{\mathbb{P}^5}(1) \to M\to 
0,$$
where $\mathfrak{S}_{\mathbb{P}^5}$ is the homogeneous coordinate ring of
$\mathbb{P}^5$.
Observe that for such modules $c_1(Syz^1(M))=-3$ and $\operatorname{rk}
(Syz^1(M))=11$. It follows by Formula (\ref{formula for canonical}) for the canonical class and the 
formula for the degree (see Lemma \ref{lem degree formula}) that if
$\bigwedge^2(Syz^1(M))(1)$ admits a
section defining a smooth
Pfaffian variety $D$, then $D$ must be a del Pezzo surface of degree 8 in its 
anti-canonical
embedding (composed with a projection).
Moreover, by Remark \ref{rem tylko dodatnie skladniki proste dla del pezzo}, if
the shifted Hartshorne--Rao module  of a
smooth surface $D$ is isomorphic to some $M$ as above then $D$ is defined by a
Pfaffian variety associated to
$Syz^1(M)$.
To such a minimal presentation of $M$ one associates an
embedding
$$\mathbb{P}^{13}\to \mathbb{P}^{17}=\langle \mathbb{P}^2\times
\mathbb{P}^5\rangle.$$

In this case, the intersection $\mathbb{P}^{13}\cap (\mathbb{P}^2\times
\mathbb{P}^5)=\mathbb{P}(\mathcal{G})$ can be seen as the projectivization of a
sheaf $\mathcal{G}$ on $\mathbb{P}^2$ given by the cokernel of the
embedding $$4\mathcal{O}_{\mathbb{P}^2}(-1)\to 6\mathcal{O}_{\mathbb{P}^2}$$
corresponding to the four
linear equations defining the $\mathbb{P}^{13}$.

 We can now adapt the notation from Section \ref{sec Tonoli 17} to the case of 
del Pezzo
surfaces.
$\mathcal{M}^D_k=\{\mathbb{P}\in G(14,18) \mid \quad \coker (\mathcal{G} \to
\mathcal{G}^{\vee\vee}) \text{ is the structure sheaf of a scheme of length 
$k$}\}$.

\begin{prop} \label{wiazki bidualne dla del pezzo}
If $\mathbf{P}\in G(14,18)$ is a $\mathbb{P}^{15}$ in $\mathbb{P}(V_3\otimes
V_6)$
then $\mathbf{P}\in \mathcal{M}^D_k$ if and only if
$\mathcal{G}_{\mathbf{P}}^{\vee\vee}$ is a rank two vector bundle isomorphic to:
\begin{enumerate}
\item $\mathcal{O}_{\mathbb{P}^2}(2)\oplus \mathcal{O}_{\mathbb{P}^2}(2)$,
\item $\mathcal{O}_{\mathbb{P}^2}(1)\oplus \mathcal{O}_{\mathbb{P}^2}(3)$,
\end{enumerate}
 for $k=6$, $7$ respectively.
\end{prop}
\begin{proof}The proof is completely analogous to the proof of Proposition 
\ref{wiazki w kojadrze dla CY stopnia 17}.
\end{proof}
Now, still analogously to the case of Tonoli Calabi--Yau threefolds, we deduce 
the following corollaries.

\begin{cor}
 In the notation of Proposition \ref{wiazki bidualne dla del pezzo},
if
$\mathbf{P}\in \mathcal{M}^D_k$ then there exists a
 map $\beta^D_{\mathbf{P}}: 7\mathcal{O}_{\mathbb{P}^2}\to L^D_k$ surjective 
outside possibly a set of codimension at least 2, such that $\beta^D_{\mathbf{P}}\circ 
\lambda^D_{\mathbf{P}} =0$ and $L^D_k$
 is one of the following sheaves:
 \begin{enumerate}
  \item $L^D_{7}=\mathcal{O}_{\mathbb{P}^2}(1)$,
  \item $L^D_{6}=\mathcal{O}_{\mathbb{P}^2}(2)$.
 \end{enumerate}
\end{cor}

\begin{cor} \label{construction via containing veronese k=6,7} Let $k\in 
\{6,7\}$. If there exists a surjective  map $\beta^D_{\mathbf{P}}:
7\mathcal{O}_{\mathbb{P}^2}\to L^D_k$ such that $\beta^D_{\mathbf{P}}\circ 
\lambda_{\mathbf{P}} =0$, then
$\mathbf{P}\in \mathcal{M}_k$.
\end{cor}

\begin{prop} \label{prop par count F1} The Tonoli families 
$\mathcal{T}_{k,14,3,5}$ of Pfaffian varieties for $k=6,7$ are mapped via the 
forgetful map to open subsets in the two components
of the Hilbert scheme of del Pezzo surfaces of degree 8 in $\mathbb{P}^5$ 
representing the two types $\mathbb{P}^1\times \mathbb{P}^1$ and $\mathbb{F}_1$ respectively.
\end{prop}
\begin{proof} By the formulas for degree and canonical class and by 
semicontinuity
it is clear that in this construction we obtain a family of del Pezzo surfaces. 
We then use Proposition
\ref{wymiar obrazu w schemat Hilberta} to compute the dimension of the family of 
constructed surfaces inside the Hilbert scheme and compare
it with the dimension of the Hilbert scheme of del Pezzo surfaces. 
We obtain that our family gives a component of the Hilbert scheme. We finish
the proof by checking one example in each case using \cite{KKS}.
\end{proof}

\begin{rem}
Geometrically, to construct a vector bundle defining a general del Pezzo surface
of type $\mathbb{P}^1\times \mathbb{P}^1$, one considers a general
$\mathbb{P}^{13}\subset \mathbb{P}^{17}$ containing the graph of a second
Veronese embedding of the projective plane in $\mathbb{P}^2\times \mathbb{P}^5
\subset \mathbb{P}^{17}$.  To construct a vector bundle
defining  a general del Pezzo surface of type $\mathbb{F}_1$, one considers
a general $\mathbb{P}^{13}\subset \mathbb{P}^{17}$ containing the graph of a
linear embedding $\mathbb{P}^2 \to \mathbb{P}^5$ in $\mathbb{P}^2\times 
\mathbb{P}^5 \subset
\mathbb{P}^{17}$.
\end{rem}
\begin{rem}
Observe that a general
$\mathbb{P}^{13}\subset \mathbb{P}^{17}$ containing the graph of the second
Veronese embedding in $\mathbb{P}^2\times \mathbb{P}^5 \subset \mathbb{P}^{17}$
contains a one-parameter family of such graphs.
\end{rem}

\section{The analogy in degrees $(8,17)$} \label{sec rel dP Cy}

Let us now finish the proof that the constructions of del Pezzo surfaces and
Calabi--Yau
threefolds of codimension $3$ are related.
\begin{proof}[Proof of Thm.~\ref{CY jako ext dP}]
It remains to handle the case of del Pezzo surface of degree $d_D=8$ and Tonoli
Calabi--Yau threefold of degree $d_X=17$. On one side we have
two families, on the other three.
Let us start with a general del Pezzo surface of degree $8$. Its shifted
Hartshorne--Rao
module
defining the bundle $E_D$ corresponds to a subspace of dimension 13 contained in
$\mathbb{P}^{17}=\langle \mathbb{P}^2\times \mathbb{P}^5\rangle$ such that the
intersection
$\mathbb{P}^{13}\cap(\mathbb{P}^2\times \mathbb{P}^5)$ contains either the graph
of a linear map $\mathbb{P}^2\to \mathbb{P}^5$ or the the graph of a second
Veronese embedding $\mathbb{P}^2\to \mathbb{P}^5$. Such a subspace is clearly the
projection of a space
$\mathbb{P}^{13}\subset  \mathbb{P}^{20}\supset \mathbb{P}^2\times \mathbb{P}^6$
with the analogous property, i.e. $\mathbb{P}^{13}\cap (\mathbb{P}^2\times
\mathbb{P}^6)$ contains either the graph of a linear map $\mathbb{P}^2\to
\mathbb{P}^6$ or the graph of a second Veronese embedding $\mathbb{P}^2\to \mathbb{P}^6$. The general such
choice of extension defines a bundle $E'_D$ on $\mathbb{P}^6$. A general
extension between this bundle and $2\mathcal{O}_{\mathbb{P}^6}$ corresponds to a
space $\mathbb{P}^{15}\subset \mathbb{P}^{20}$ containing the
$\mathbb{P}^{13}$,
i.e., $\mathbb{P}^{15}\cap (\mathbb{P}^2\times \mathbb{P}^6)$ contains
either the graph of a linear map $\mathbb{P}^2\to \mathbb{P}^6$ or the the graph
of a second Veronese embedding $\mathbb{P}^2\to \mathbb{P}^6$. In particular, it
corresponds to an element of $\mathcal{B}_{9}$ or $\mathcal{B}_{11}$ (notation as in Theorem \ref{main result}). To prove
that the corresponding bundle defines a Calabi--Yau threefold, we observe that
a general element of $\mathcal{B}_{9}$ or $\mathcal{B}_{11}$ arises in this way.
Indeed, for a general $\mathbf{P}\in \mathcal{B}_{9}$ or $\mathcal{B}_{11}$ the corresponding $L_{\mathbf{P}}$ contains the graph of a Veronese or linear embedding 
$\mathbb{P}^2\to \mathbb{P}^6$ each of which spans a space of dimension smaller than 13. Hence $L_{\mathbf{P}}$ contains a $\mathbb{P}^{13}$ also containing these graphs.
Finally the image of this $\mathbb{P}^{13}$ via a general projection
$$\mathbb{P}^2\times \mathbb{P}^6\subset \mathbb{P}^{20} \to
\mathbb{P}^{17}\supset \mathbb{P}^2\times \mathbb{P}^5$$
induced by a projection  $\mathbb{P}^6\to \mathbb{P}^5$ is a $\mathbb{P}^{13}$ defining a point in 
$ \mathcal{M}^D_6$ or in $\mathcal{M}^D_7$. 
The proposition is hence proven for
Calabi--Yau threefolds with $k=9$ or $11$.

Let us now consider the case of Calabi--Yau threefolds of
degree 17 with $k=8$. In this case we consider a $\mathbb{P}^{15}$ such that
$\mathcal{G}^{\vee\vee}=T_{\mathbb{P}^2}(1)$; the latter admits a 2-dimensional 
family
of surjections onto $I_p(3)$ parameterized by $p\in \mathbb{P}^2$. The appropriate composite
map defines a
$\mathbb{P}^{13}$ spanned by the graph of a rational map $\mathbb{P}^2
\dashrightarrow \mathbb{P}^6$ defined by a system of cubics passing through a
point. We claim that the projection of this $\mathbb{P}^{13}\subset \langle \mathbb{P}^2\times \mathbb{P}^6\rangle$ onto
$\mathbb{P}^{17}=\langle  \mathbb{P}^2\times \mathbb{P}^5\rangle$ is a general
element of $\mathcal{M}^D_6$ and hence defines a del Pezzo surface $D^1_8$.
Indeed, we just observe that the projected $\mathbb{P}^{13}$ is associated to a
map
$4\mathcal{O}_{\mathbb{P}^2}(-1)\to 6 \mathcal{O}_{\mathbb{P}^2} $ whose
cokernel admits a surjection
on $I_p(3)$. We then compute the Chern classes of this cokernel and deduce that we
are in $\mathcal{M}^D_6$.
\end{proof}
\begin{rem}\label{ASD} In Theorem \ref{CY jako ext dP}, we relate Calabi--Yau
threefolds
to del Pezzo surfaces or more precisely appropriate vector bundles defining
these varieties via the Pfaffian construction  in two
steps, passing through a vector bundle $E_D'$ on $\mathbb{P}^6$.
One might wonder if there is a variety given by Pfaffians of this bundle. By
Formula (\ref{formula for canonical}) and degree formulas such a variety if smooth would be a Fano threefold
of
index 2 and degree $d-9$ in $\mathbb{P}^6$. And indeed for Calabi--Yau
threefolds
of degree $d\leq 16$ the bundle $E_D'$ defines a family of such
smooth Fano threefolds.
For $d=17$ the situation is different. The only Fano threefolds of index 2 and
degree 8 in $\mathbb{P}^6$ are projections of second Veronese
embeddings of $\mathbb{P}^3$.
Now, using our methods, one can easily check that such a second Veronese
embedding of $\mathbb{P}^3$ in $\mathbb{P}^6$ is associated to a
$\mathbb{P}^{13}$ corresponding to a skew-symmetric map $\theta$ as in
Remark \ref{skew Calabi--Yau}. The restriction of
the associated bundle
to a general $\mathbb{P}^5$ defines a del Pezzo surface of type $\mathbb{F}_1$, whereas
a general extension bundle with $2\mathcal{O}_{\mathbb{P}^6}$ corresponding to a
$\mathbb{P}^{15}$ containing our $\mathbb{P}^{13}$ defines a Calabi--Yau
threefold from the special family discussed in Remark \ref{skew Calabi--Yau}. We hence do not recover the whole family of Tonoli
Calabi--Yau threefolds of degree $17$ with $k=9$.
In the general case for any $k=8,9,11$ the Pfaffians associated to a
general section $s\in H^0(\bigwedge^2 E_D')$ do not define a variety of expected
codimension. Only after restricting to a general $\mathbb{P}^5$, do appropriate
sections appear.
\end{rem}

\section{Conjectures and problems} \label{final 9 18}
Assuming that the relation observed in Theorem \ref{CY jako ext dP} between
Calabi--Yau threefolds in $\mathbb{P}^6$ and del Pezzo surfaces in
$\mathbb{P}^5$
follows from a more general phenomenon it is natural to make the following
conjecture.
\begin{conj}
 There are no Calabi--Yau threefolds of degree $d\geq 19$ in $\mathbb{P}^6$.
\end{conj}
In fact, we believe that a slightly more general conjecture may hold.
\begin{conj}
  There are no canonical surfaces of degree $d\geq 19$ in $\mathbb{P}^5$.
 \end{conj}

Having formulated these conjectures, the most interesting case to be studied at the moment is the case of Calabi--Yau threefolds of 
degree $18$. Since there is a del Pezzo surface of degree 9, we can try to use
it to construct Calabi--Yau threefolds of degree 18
in $\mathbb{P}^6$ or canonical surfaces of degree 18 in $\mathbb{P}^5$.
Indeed, in \cite{CYTon2}, using the construction of del Pezzo surfaces of
degree
9, we construct a family of canonical surfaces of degree 18 in $\mathbb{P}^5$.
We have so far been unable to find the description of a general
such surface and unable to prove the existence of Calabi--Yau threefolds of
degree 18 in $\mathbb{P}^6$. It seems that, to solve any of these two problems,
the key is to find a geometric classification of all bundles on
$\mathbb{P}^5$ which define, by the Pfaffian construction, del Pezzo surfaces of
degree 9 in $\mathbb{P}^5$. We plan to address this problem in a subsequent
paper.
\bibliographystyle{alpha}
\bibliography{biblio}
\vskip1cm 

\begin{minipage}{15cm}
 Department of Mathematics and Informatics,\\ Jagiellonian
University, {\L}ojasiewicza 6, 30-348 Krak\'{o}w, Poland.\\
\end{minipage}

\begin{minipage}{15cm}
Institute of Mathematics of the Polish Academy of Sciences,\\
ul. \'{S}niadeckich 8, P.O. Box 21, 00-956 Warszawa, Poland.\\
\end{minipage}

\begin{minipage}{15cm}
Institut f\"ur Mathematik\\
Mathematisch-naturwissenschaftliche Fakult\"at\\
Universit\"at Z\"urich, Winterthurerstrasse 190, CH-8057 Z\"urich\\
\end{minipage}

\begin{minipage}{15cm}
\emph{E-mail address:} grzegorz.kapustka@uj.edu.pl\\
\emph{E-mail address:} michal.kapustka@uj.edu.pl
\end{minipage}

\end{document}